\documentclass[11pt,oneside,leqno]{article}

\usepackage[T1]{fontenc}
\usepackage{mathptmx}
\usepackage{amsthm,amsmath,amsfonts,amssymb,mathtools}
\usepackage[numbers]{natbib}
\usepackage{booktabs,array}
\usepackage{enumitem}
\usepackage{graphicx}
\usepackage{microtype}
\usepackage[hidelinks]{hyperref}
\usepackage[margin=1.05in]{geometry}

\hypersetup{
  pdftitle={Sharp Berry--Esseen Bounds for the Log Determinant of a Gaussian Sample Correlation Matrix},
  pdfauthor={Hongru Zhao}
}

\newtheorem{theorem}{Theorem}[section]
\newtheorem{proposition}[theorem]{Proposition}
\newtheorem{lemma}[theorem]{Lemma}
\newtheorem{corollary}[theorem]{Corollary}

\theoremstyle{definition}

\numberwithin{equation}{section}

\newcommand{\E}{\mathbb E}
\newcommand{\Pp}{\mathbb P}
\newcommand{\R}{\mathbb R}
\newcommand{\N}{\mathcal N}
\newcommand{\tr}{\operatorname{tr}}
\newcommand{\diag}{\operatorname{diag}}
\newcommand{\Var}{\operatorname{Var}}

\newcommand{\Log}{\operatorname{Log}}

\newcommand{\cA}{\mathcal A}
\newcommand{\DK}{d_{\mathrm K}}
\newcommand{\abs}[1]{\lvert#1\rvert}

\newcommand{\Znull}{Z_{0,m,p}}
\newcommand{\ZR}{Z_{R,m,p}}
\makeatletter
\newcommand{\bdoi}[1]{\href{https://doi.org/#1}{\nolinkurl{https://doi.org/#1}}}
\providecommand{\bibfont}{\normalsize}
\newenvironment{frontmatter}{}{}
\newenvironment{aug}{}{}

\newcommand{\runtitle}[1]{}

\newcommand{\printead}[2][]{\unskip}

\makeatother

\begin{document}
\hypersetup{pdfsubject={Submitted to Probability Theory and Related Fields}}
\raggedbottom
\hypersetup{pageanchor=false}

\begin{frontmatter}
\title{Sharp Berry--Esseen Bounds for the Log Determinant of a Gaussian
Sample Correlation Matrix}
\runtitle{Sharp Berry--Esseen bounds for a correlation log determinant}

\begin{aug}
\author{Hongru Zhao}
\date{School of Statistics, University of Minnesota, Minneapolis,
Minnesota, USA\\\href{mailto:zhao1118@umn.edu}{zhao1118@umn.edu}}
\maketitle
\end{aug}

\begin{abstract}
Let $\widehat R$ be the Pearson sample correlation matrix formed from $n$
independent Gaussian observations in $p$ dimensions, and write $m=n-1\ge p$.
Under the null correlation $R=I_p$, the classical independent beta product,
exact cumulants, and full Fourier inversion yield, along every sequence
$p\to\infty$ with $m\ge p$, a uniform first Edgeworth expansion for
$\log\det\widehat R$, centered by its exact mean and scaled by its exact
standard deviation. The expansion identifies the exact finite dimensional
skewness correction and gives the sharp Kolmogorov equivalent
$A_{m,p}/\{6\sqrt{2\pi}\,V_{m,p}^{3/2}\}$, where $V_{m,p}$ is the exact variance
and $A_{m,p}$ is the absolute third cumulant. This equivalent unifies the
square, fixed gap, growing gap, proportional, and dilute regimes; in the square
regime the error has order $(\log p)^{-3/2}$ with an exact constant. For every
positive definite population correlation matrix $R$, we prove a uniform
finite sample Berry--Esseen bound that explicitly tracks population dependence.
All theoretical results have exact or proved equivalent Lean 4 formulations whose declarations and dependencies are kernel checked.
\end{abstract}

\end{frontmatter}

\pagenumbering{arabic}
\setcounter{page}{1}
\hypersetup{pageanchor=true}

\section{Introduction}

A correlation matrix records how strongly several measured quantities move
together.  The determinant of a correlation matrix has a geometric meaning:
the determinant is the squared volume
of a parallelepiped built from standardized data vectors.  A determinant near
one means that the directions are close to orthogonal; a determinant near
zero means that some directions are nearly redundant.  Taking a logarithm
turns products of squared projection factors into sums, which is the natural setting
for a central limit theorem.  The determinant has consequently appeared in
tests of independence and in high dimensional correlation inference.
Moreover, a sample size multiple of $-\log\det\widehat R$ is the Gaussian
likelihood ratio statistic for complete independence, while
$-\tfrac12\log\det\widehat R$ is the empirical analogue of Gaussian total correlation
\cite{Wilks1932,Watanabe1960,RoweDay2019}; see also
Muirhead~\cite[Chapter~8]{Muirhead1982}, Jiang and
Qi~\cite{JiangQi2015}, and Jiang~\cite{Jiang2019}.  Finite-dimensional
determinant laws and scale free scatter representations were studied in
\cite{GuptaRathie1983,MathaiProvost2024,SebastianPrincy2024}.

The same determinant is a squared random volume.  This connects the problem
to spherical matrix models and high dimensional random simplex geometry
\cite{Dawid1977,GusakovaHeinyThale2023,ShanLi2026}.  It also distinguishes
the present Pearson statistic from a covariance determinant: random
coordinatewise normalization is part of the object rather than a removable
scale factor.

The principal result is a uniform Berry--Esseen theorem with an explicit
first order asymptotic constant for the standardized log determinant.  Under
$R=I_p$, Theorem~\ref{thm:null-edgeworth} proves an
exact first order asymptotic for the Kolmogorov distance, uniformly over every
admissible sequence $p\to\infty$ with $m\ge p$.  For arbitrary $R$, Theorem
\ref{thm:general-full} supplies a uniform finite bound with an explicit
separation between the Wishart leading term and the nonlinear
diagonal standardization remainder.

There is one complication.  The quality of the normal approximation depends
strongly on the residual dimension gap $m-p$.  When the sample size is
well above the dimension, many small effects average out.  When the residual
dimension gap is small, the last few squared Gram--Schmidt heights have
substantial mass near zero, and their logarithms retain appreciable skewness.
Thus no answer based only on a
fixed limiting aspect ratio $p/m\to\gamma$ can describe the whole parameter
range.
The terminology reflects the Marchenko--Pastur lower spectral edge and its
Bessel type boundary behavior \cite{MarchenkoPastur1967,TracyWidom1994}.
Near-square covariance log determinants have related, but different,
asymptotics because they do not contain Pearson normalization
\cite{WangHanPan2018}.

The argument avoids choosing an asymptotic regime too early. We first derive
the exact finite dimensional third cumulant in
\eqref{eq:A-lambda}.  Only after the probability approximation is
complete do we simplify \eqref{eq:unified-equivalent} in the square,
proportional, or dilute limits.  The order of work just described
prevents endpoint terms from being accidentally discarded.  The exact beta
product underlying the reduction is classical in random Gram matrix theory;
Rouault~\cite{Rouault2007} gives precise finite dimensional statements, and
Heiny, Johnston and Prochno~\cite[Theorem~A]{HJP2022} give a published
all regime normal approximation bound after the parameter identification in
Appendix~\ref{app:beta}.

The second philosophy is that a sharp Berry--Esseen result needs signed
information.  A classical inequality adds absolute third moments.  The
absolute moment inequality is safe but destroys cancellation inside the
nearly Gaussian factors and can
lose a square root of the dimension.  The signed third cumulant gives the
correct first departure from normality.  To turn the signed local expansion into a
Kolmogorov result, however, every Fourier frequency must be controlled.  The
exact beta product supplies the required global control.  Xie and
Sun~\cite[Theorem~2.1]{XieSun2021} prove higher order Edgeworth expansions in
the proportional regime. The present paper establishes the endpoint uniform
sharp Kolmogorov equivalent when the aspect ratio may also approach one.

For general population correlation, the diagonal standardization introduces
a dependent nonlinear error.  We separate an analytically tractable leading
term from the nonlinear error.  The resulting upper bound is uniform but not
declared sharp.  The distinction between a proved upper bound and a sharp
equivalent is maintained throughout.  The decomposition is organized by a
Wiener chaos expansion of the nonlinear diagonal terms and is developed from
the Gaussian sample correlation identities of Zhao~\cite{Zhao2026}, while the
matrix gamma and Wishart inputs are taken exactly from
Muirhead~\cite[Theorems~2.1.11 and 3.2.1]{Muirhead1982}.
The Gaussian expansion belongs to the classical Wiener chaos framework
\cite{Wiener1938,Janson1997,NourdinPeccati2012}; related determinant
applications include Notarnicola and Diez--Tudor
\cite{Notarnicola2023,DiezTudor2023}.

\subsection{Contributions and scope}

The paper makes five mathematical contributions.  First, the beta product
representation is combined with exact centering and scaling to produce the
finite dimensional cumulant ratio in \eqref{eq:unified-equivalent}.  Second,
a model specific characteristic function argument with full frequency Fourier
inversion turns the signed cubic term into the uniform expansion in
\eqref{eq:null-edgeworth}; control only near frequency
zero would not suffice.  Third, Proposition~\ref{prop:asymptotics} evaluates
the exact cumulant ratio uniformly across the hard, proportional, and dilute
regions.  Fourth, Corollary~\ref{cor:square-supremum} proves that the square
model is asymptotically worst after optimization over the full admissible
domain.  Fifth, Theorems~\ref{thm:general-leading} and
\ref{thm:general-full} give a two stage bound for arbitrary population
correlation.

The paper also makes the scope of its formalization explicit.  Section~\ref{sec:lean}
states which probability claims were encoded in Lean and which judgments
remain outside the kernel.  The distinction is important because three kinds
of correctness are involved.  Mathematical correctness asks whether a
conclusion follows from assumptions.  Translation correctness asks whether
the formal definitions express the intended sample correlation problem.
Provenance correctness asks whether a cited source states exactly the
attributed result.  The Lean development imports no literature result as an
unproved source statement: every formal dependency is proved in the project
or supplied by kernel checked mathlib.  Consequently, bibliographic
provenance cannot introduce an error into the kernel checked theorem.
Translation correctness and scholarly attribution are audited separately in
Appendix~\ref{app:ledger}.

The scope of these results should be kept visible.  The statistic is exactly centered
and scaled; changing the normalization can change a sharp first order
constant.  The null expansion is uniform over $m\ge p$, but the bound for
arbitrary $R$ is not asserted to be sharp.  Theorems~\ref{thm:null-edgeworth} and
\ref{thm:general-full} concern Gaussian
observations; the exact beta product and Wishart transform are Gaussian
structures.  Finally, Theorems~\ref{thm:null-edgeworth} and
\ref{thm:general-full} control Kolmogorov distance, so the two theorems do
not automatically provide a relative error approximation for extremely
small tail probabilities.

\subsection{How to read the notation}

Only three scales drive the main conclusions.  The quantity $V_{m,p}$ in
\eqref{eq:null-variance} is the exact null variance.  The quantity $A_{m,p}$
in \eqref{eq:A-lambda} is the absolute value of the signed third cumulant.
The gap $d=m-p$ records the number of residual directions left after the last
sample variable is added.  Equation~\eqref{eq:unified-equivalent} should be
read first in terms of $A_{m,p}$ and $V_{m,p}$; the gap $d$ enters only when
Proposition~\ref{prop:asymptotics} simplifies the exact sums.

The symbols $M_R$ and $E_R$ appear only in the arbitrary correlation part.
Equation~\eqref{eq:general-decomposition} defines $M_R$ as the analytically
tractable leading variable and $E_R$ as the centered nonlinear remainder.
The scale $s_R$ in \eqref{eq:general-scale} combines the null variance with
the population correlation energy.  Section~\ref{sec:generalR} repeatedly
returns to the decomposition in \eqref{eq:general-decomposition}, so the
reader need not retain the coordinate formulas on a first reading.

\subsection{Why a classical Berry--Esseen inequality is not enough}

For a sum of independent centered variables, a classical Berry--Esseen
inequality controls Kolmogorov distance by a sum of absolute third moments
divided by the variance to the three halves power.  Proposition
\ref{prop:beta-product} makes such an inequality available in principle under
$R=I_p$.  The resulting quantity is safe as an upper bound, but the absolute
value is taken before contributions from different beta factors are combined.

The sharp asymptotic problem is sensitive to the order of signed summation and
the final absolute value.
The exact logarithmic cumulant in \eqref{eq:exact-cumulants} adds signed
third order contributions first, and only then takes an absolute value through
$A_{m,p}$ in \eqref{eq:A-lambda}.  The cancellation is particularly important
away from the hard edge.  A generic absolute moment bound can therefore have
the wrong order even though the beta factors are independent and every
classical hypothesis is satisfied.

The signed calculation still needs a global error bound.  A formal Taylor
expansion of the characteristic function identifies the candidate coefficient
but does not prove that the candidate coefficient dominates all uncomputed terms after
Fourier inversion.  Theorem~\ref{thm:null-edgeworth} combines the signed local
coefficient with model specific frequency decay.  The combination, rather
than the cubic Taylor term alone, is what upgrades a heuristic Edgeworth
correction to the exact Kolmogorov equivalent in
\eqref{eq:unified-equivalent}.

The distinction also explains the two part structure of the paper.  Section
\ref{sec:null} states a probability approximation while preserving exact
polygamma quantities.  Sections~\ref{sec:rates}--\ref{sec:regimes} evaluate
the exact polygamma quantities under different dimension constraints.
Combining the two
steps too early would entangle a Fourier remainder with regime specific
algebra and would make the hard endpoint difficult to audit.

The exposition develops the probability and special function ingredients
before the technical proofs.  The appendices assume familiarity with
characteristic functions, analytic methods, Gaussian expansions, and Wishart
calculations.
Section~\ref{sec:background} introduces the model and the small
amount of special function language.  Section~\ref{sec:prior} states exactly
what is borrowed from the literature.  Sections~\ref{sec:null} and
\ref{sec:generalR} give the main results and the corresponding proof ideas.
Section~\ref{sec:rates}
derives a unified formula, and Section~\ref{sec:regimes} specializes the
formula under common dimension gap constraints. Section~\ref{sec:applications}
explains the statistical meaning of the bounds. Section~\ref{sec:lean}
records the formal verification boundary. The appendices contain the proofs
needed to make the main paper self contained, together with the statement level
Lean crosswalk. The thirteen declarations in
Appendix~\ref{app:ledger} are the principal public theorem endpoints, not an
inventory of every supporting declaration in their dependency cone. Lean
formalization covers, directly or through proved finite sample bridges,
Proposition~\ref{prop:beta-product}, Theorems~\ref{thm:null-edgeworth},
\ref{thm:general-leading}, and \ref{thm:general-full}, all clauses of
Proposition~\ref{prop:asymptotics}, and
Corollary~\ref{cor:square-supremum}. Here a proved finite sample bridge is an
exact Lean identity of functions or probability laws, valid for every
admissible finite (m,p), that connects the statistic stated in the paper to
the canonical random variable used in the formal proof. It therefore proves
that the formal endpoint is equivalent to the paper statement, rather than
merely asymptotically related to it. Appendix~\ref{app:ledger} lists the exact
or proved equivalent relations. Kernel checking applies to formal declarations
and their dependencies; the mathematical prose of the article
is not itself kernel checked.

\section{Model and mathematical background}\label{sec:background}

\subsection{The sample correlation matrix}

Let $x_1,\ldots,x_n$ be independent $N_p(\mu,\Sigma)$ observations.  Define
the scatter matrix and the corresponding correlation normalization by
\begin{align}
 S&=\sum_{k=1}^n(x_k-\bar x)(x_k-\bar x)^\top,\label{eq:scatter}\\
 \widehat R&=\diag(S)^{-1/2}S\diag(S)^{-1/2}.\label{eq:sample-corr}
\end{align}
The population correlation matrix is
\begin{equation}
 R=\diag(\Sigma)^{-1/2}\Sigma\diag(\Sigma)^{-1/2}.
 \label{eq:population-corr}
\end{equation}
We write $m=n-1$.  Throughout, $p\to\infty$, $R$ is positive definite, and
$m\ge p$.  The last inequality is not cosmetic: if $p>m$, then $S$ has rank
at most $m$, so $\det\widehat R=0$ almost surely and an ordinary finite log
determinant does not exist. Thus $m\ge p$ is necessary for the ordinary
finite log determinant to exist, with $m=p$ as the common endpoint.

The loss of one degree of freedom explains the notation $m=n-1$.  Subtracting
the sample mean projects every data column onto the subspace of dimension $m$
orthogonal to the all ones vector.  Under $R=I_p$, the projected columns are
independent isotropic Gaussian vectors in the centered data subspace.
Dividing each projected column by the Euclidean norm of the projected column
removes the sample variance and leaves
an independent uniform direction.  The matrix in \eqref{eq:sample-corr} is
therefore a Gram matrix of random directions.  Appendix~\ref{app:beta}
turns the geometric description into the distributional equality in
\eqref{eq:beta-product}.

The Gram interpretation also explains the endpoint $m=p$.  At the square
endpoint, the final direction has only one new orthogonal coordinate
available.  The final squared height can consequently approach zero much
more strongly than a typical squared projection factor in a dilute design.
Taking logarithms magnifies these small final factors.  The variance grows, but the standardized
sum retains enough skewness to make the normal approximation converge at a
logarithmic rather than polynomial rate.  Corollary
\ref{cor:square-supremum} converts the geometric intuition into an exact
worst case statement.

For $R\ne I_p$, the centered columns cease to be independent across variable
index, although the scatter matrix still has a Wishart representation.
Correlation normalization divides the Wishart matrix by the random diagonal
of the Wishart matrix.
The log determinant therefore contains a log Wishart determinant and a sum
of dependent log diagonal terms.  Equation~\eqref{eq:general-decomposition}
is designed to isolate the linear part of the dependent log diagonal terms
before the normal approximation is attempted.

\subsection{Kolmogorov distance and exact normalization}

For random variables $X,Y$, the Kolmogorov distance is
\begin{equation}
 \DK(X,Y)=\sup_{x\in\R}\abs{\Pp(X\le x)-\Pp(Y\le x)}.
 \label{eq:kolmogorov}
\end{equation}
Let $\Phi$ and $\phi$ be the standard normal distribution and density.

The gamma function extends the factorial.  The logarithmic derivatives of
the gamma function are
the digamma and polygamma functions:
\begin{equation}
 \psi(x)=\frac{\Gamma'(x)}{\Gamma(x)},\qquad
 \psi_r(x)=\frac{\mathrm d^{r+1}}{\mathrm d x^{r+1}}\log\Gamma(x).
 \label{eq:polygamma-definition}
\end{equation}
Under $R=I_p$, Guerrero~\cite{Guerrero1994} gave an early moment
calculation; the corrected exact mean and variance below are due to Rowe and
Day~\cite{RoweDay2019}:
\begin{align}
 b_{m,p}
  &=\sum_{j=2}^p\left\{
       \psi\!\left(\frac{m-j+1}{2}\right)-\psi\!\left(\frac m2\right)
     \right\},\label{eq:null-mean}\\
 V_{m,p}
  &=\sum_{j=2}^p\left\{
       \psi_1\!\left(\frac{m-j+1}{2}\right)-\psi_1\!\left(\frac m2\right)
     \right\}.\label{eq:null-variance}
\end{align}
Equations~\eqref{eq:null-mean} and \eqref{eq:null-variance} are derived from
the beta product in Appendix~\ref{app:beta}.  Under
$R=I_p$, set
\begin{equation}
 \Znull=\frac{\log\det\widehat R-b_{m,p}}{\sqrt{V_{m,p}}}.
 \label{eq:null-statistic}
\end{equation}
Exact centering and scaling matter: an error of the same order as the first
Edgeworth correction can change the leading Kolmogorov constant.

Kolmogorov distance has a direct calibration interpretation.  If
$\DK(X,\N(0,1))\le\delta$, then the probability assigned by the normal
approximation on every interval of the form $(-\infty,x]$ differs from the exact probability by
at most $\delta$.  In particular, normal critical values have absolute size
error at most $\delta$.  The metric is strong enough to compare distribution
functions uniformly, while remaining compatible with characteristic function
smoothing.  Theorem~\ref{thm:null-edgeworth} goes beyond an inequality by
identifying the asymptotic value of the worst CDF discrepancy.

Exact normalization serves two roles.  The exact center $b_{m,p}$ removes the
first cumulant of the log beta sum, and the exact variance $V_{m,p}$ makes the
second cumulant equal to one.  More subtly, exact normalization prevents a
deterministic location or scale error from competing with the cubic
Edgeworth term.  For a qualitative central limit theorem, a center differing
by $o(\sqrt{V_{m,p}})$ can be sufficient.  For the equivalent in
\eqref{eq:unified-equivalent}, the permissible difference is smaller because
the target Kolmogorov error itself tends to zero.

The scale $s_R$ for general $R$ in \eqref{eq:general-scale} has the same design
principle.  The first term is the null log determinant variance, and the
second term is the variance contributed by the linear population correlation
fluctuation.  Equations~\eqref{eq:scale-mismatch} and
\eqref{eq:tau-comparison} quantify, rather than merely assume, the relation
between $s_R$ and the exact variances arising in the proof.

For general $R$, let
\begin{equation}
 A_R=R-I_p,\qquad a_R=\tr(A_R^2),\qquad
 s_R^2=V_{m,p}+\frac{2a_R}{m},
 \label{eq:general-scale}
\end{equation}
and define
\begin{equation}
 \ZR=\frac{\log\det\widehat R-\log\det R-b_{m,p}}{s_R}.
 \label{eq:general-statistic}
\end{equation}

\subsection{Cumulants and the first Edgeworth correction}

If $K_X(z)=\log\E e^{zX}$ is analytic near zero, then the coefficient of
$z^r/r!$ is the $r$th cumulant $\kappa_r(X)$.  The first two cumulants are the
mean and variance.  For a centered variance one variable, the third cumulant
measures signed skewness.  The first formal correction to the normal CDF is
\begin{equation}
 \Phi(x)+\frac{\kappa_3}{6}(1-x^2)\phi(x).
 \label{eq:edgeworth-philosophy}
\end{equation}
Appendix~\ref{app:fourier} proves the Fourier identity in
\eqref{eq:edgeworth-philosophy} and, more importantly, proves that the
remainder is uniformly smaller in the present problem.

The phrase ``first Edgeworth correction'' can be understood without using a
general Edgeworth theorem.  After exact standardization, the logarithm of the
characteristic function begins with the Gaussian quadratic term.  The next
term is cubic and has coefficient equal to the standardized third cumulant.
Fourier inversion maps the cubic frequency term to the polynomial correction
in \eqref{eq:edgeworth-philosophy}.  The sign of the third cumulant is retained
throughout the Fourier calculation.  Retaining the sign is the reason that
$A_{m,p}$, rather than a sum of absolute third moments of the individual
log beta variables, appears in \eqref{eq:unified-equivalent}.

The local expansion alone is not a distributional theorem.  Fourier
inversion integrates over an unbounded frequency range, so a Taylor formula
valid near zero leaves the middle and high frequencies uncontrolled.  The
proof of Theorem~\ref{thm:null-edgeworth} uses the gamma product identity
\eqref{eq:gamma-modulus-product} to establish the local Gaussian damping
estimate \eqref{eq:null-local-fourier} and the complementary power decay
estimate \eqref{eq:null-tail-fourier}.  Section~\ref{sec:null} gives the proof
architecture, and Appendix~\ref{app:fourier} supplies the quantified
frequency estimates.

\subsection{Three levels of normal approximation}

A central limit theorem, a Berry--Esseen bound, and an Edgeworth expansion
answer different questions.  A central limit theorem asserts that the
Kolmogorov distance in \eqref{eq:kolmogorov} tends to zero, but a central
limit theorem need not state a rate.  A Berry--Esseen bound supplies an
explicit upper envelope for the distance.  An Edgeworth expansion identifies
a signed correction to the Gaussian CDF and a remainder smaller than the
signed correction.

The published bound \eqref{eq:HJP} is a Berry--Esseen statement in the
second sense: it gives a finite upper bound with constant $28$.
Theorem~\ref{thm:null-edgeworth} is an Edgeworth statement in the third
sense, and \eqref{eq:unified-equivalent} converts the signed expansion into a
sharp Berry--Esseen asymptotic.  Theorem~\ref{thm:general-full} is an upper
bound for arbitrary $R$; no signed equivalent for arbitrary $R$ is claimed.

The distinctions prevent two overstatements.  A rate upper bound cannot be
called an exact asymptotic unless a matching lower bound is proved.  Likewise,
a pointwise Edgeworth formula cannot be promoted to Kolmogorov distance
without a remainder uniform in the threshold.  The lower bound in Theorem
\ref{thm:null-edgeworth} comes from the signed profile at $x=0$, and the
uniformity comes from full Fourier inversion.

The triangular array aspect creates a second layer of uniformity.  For each
$p$, the distribution contains $p-1$ beta factors whose parameters depend on
$m$.  The number of summands, the smallest analytic radius, and the variance
all change together.  A theorem uniform only over $x$ at one fixed aspect
ratio is therefore different from Theorem~\ref{thm:null-edgeworth}, which is
uniform over both the threshold and every admissible relation between $m$ and
$p$.

\section{Prior work and provenance}\label{sec:prior}

Section~\ref{sec:prior} separates borrowed facts from results proved in the
present paper.

\medskip\noindent\textit{Exact beta product.}
Rouault~\cite{Rouault2007} gives the beta decomposition for the successive
Gram squared projection factors of independent uniform directions and the corresponding
determinant product.  Setting Rouault's ambient
dimension equal to $m$ and shifting the index gives exactly the product in
Proposition~\ref{prop:beta-product}.  Rowe and
Day~\cite{RoweDay2019} give an
equivalent exact factorization for Gaussian total correlation after
identifying Rowe and Day's uncentered Wishart degrees of freedom with the
residual degrees of freedom $m$ used here.  We also include a short Gram--Schmidt proof in
Appendix~\ref{app:beta}, so no unproved transfer between models is needed.

\medskip\noindent\textit{All-regime upper bound.}
The spherical variable of Heiny, Johnston, and
Prochno~\cite{HJP2022} is one half of the null log determinant plus a
constant.  Centering and standardization remove both changes.  Their paper
therefore gives the following exact specialization.

For $p\ge41$ and $m\ge p$, put $\theta=(p-1)/m$. Their result gives
\begin{equation}
 \DK(\Znull,\N(0,1))
 \le
 \frac{28\theta^2}
 {m(1-\theta)\{\log(1/(1-\theta))-\theta\}^{3/2}}.
 \label{eq:HJP}
\end{equation}
Appendix~\ref{app:beta} records only the variable and parameter identification
needed to quote this published bound; it does not reprove the inequality.
The result is a finite sample bound with an explicit constant. The upper bound in
\eqref{eq:HJP} has the correct orders in all regimes
listed in Table~\ref{tab:regime-rates}, but the number $28$ is not a sharp
first order constant.

\medskip\noindent\textit{Proportional Edgeworth literature.}
Xie and Sun~\cite{XieSun2021} give higher order Edgeworth expansions under
$p/m\to\gamma\in(0,1)$ with the same exact polygamma normalization, together
with a computable uniform error bound. Their first correction agrees with
\eqref{eq:edgeworth-philosophy}. The present full frequency argument proves the
endpoint uniform sharp Kolmogorov equivalent and matching lower bound stated in
Theorem~\ref{thm:null-edgeworth}.

\medskip\noindent\textit{General population correlation.}
Jiang proves qualitative central limit theorems for sample correlation
determinants under stated spectral and asymptotic conditions
\cite{Jiang2019}.  Zhao develops a coordinatewise Wiener chaos expansion of
the nonlinear diagonal normalization and derives the exact decomposition and
variance identities used as the starting point here~\cite{Zhao2026}.  For
completeness and to avoid depending on
unchecked intermediate steps, Appendices
\ref{app:general-decomp}--\ref{app:wishart} rederive every imported identity
used in Theorems~\ref{thm:general-leading} and \ref{thm:general-full}.

\medskip\noindent\textit{High dimensional correlation determinants.}
Null likelihood ratio calibrations, refinements, and moderate deviations are
developed in \cite{JiangYang2013,QiWangZhang2019,HuQi2023,BaiZhangLi2024}.
Linear spectral statistic methods for correlation matrices cover important
proportional or structurally regular regimes
\cite{GaoHanPanYang2017,MestreVallet2017,YinLiTianZheng2022,YinZhengZou2023,ChenZhengZou2026}.
Li allows unbounded population spectra at the level of limiting spectral
distributions and spikes \cite{Li2025}.  Complementary non-Gaussian results
for correlation log determinants and their null limits appear in
\cite{HeinyParolya2024,ParolyaHeinyKurowicka2024,LiPanXieZhou2024,LiLiuXieZhou2026}.
These results clarify which aspects of the present theorem depend on
Gaussian centering, all regime uniformity, or the arbitrary population
correlation matrix.

\medskip\noindent\textit{Adjacent normal approximation tools.}
Song~\cite{Song2020} gives uniform Edgeworth expansions for
independent arrays with positive definite covariance matrices, a uniform
bound on moments of order $s$, and a mean weak Cram\'er condition with common
parameters.  Derumigny, Girard, and
Guyonvarch~\cite{DGG2024} give explicit bounds to a first
Edgeworth approximation for standardized sums of independent centered real
variables with finite fourth moments.  Dal Borgo,
Hovhannisyan, and Rouault study mod-Gaussian convergence for random
determinants, including uniform Gram ensembles~\cite{DBHR2019}.  D\"oring,
Jansen, and Schubert survey cumulant normal approximation~\cite{DJS2022}.
Billingsley and Petrov provide standard background on weak convergence and
triangular array limit theory \cite{Billingsley1999,Petrov1995}.
The results cited in the preceding four sentences guide the proof but are not
cited as if the cited results already contained the all regime equivalent in
\eqref{eq:unified-equivalent}.

\medskip\noindent\textit{Covariance log determinants and spectral CLTs.}
The broader random matrix literature treats \(\log\det\) as the linear
spectral statistic associated with \(f(x)=\log x\). Bai and Silverstein
\cite{BaiSilverstein2004} prove a central limit theorem for smooth linear
spectral statistics of large dimensional sample covariance matrices in the
interior aspect ratio regime. Cai, Liang, and Zhou
\cite{CaiLiangZhou2015} specialize the covariance problem to the log
determinant and connect it to differential entropy estimation. These results
are important benchmarks, but a sample correlation matrix is obtained after
random diagonal normalization, and the logarithm becomes singular at the
hard spectral edge. The exact beta reduction used here is therefore not a
cosmetic alternative to a covariance matrix spectral CLT: it retains the
center and variance at finite \((m,p)\), reaches \(m=p\), and exposes the signed
third cumulant needed for the sharp Kolmogorov constant. Conversely, the
present theorem is not a replacement for a general linear spectral statistic
CLT, because it uses the Gaussian correlation determinant's special
factorization and global gamma product decay.

\subsection{Why the distinctions in Section~\ref{sec:prior} matter}

A nearby theorem is not automatically an exact source.  Four differences are
especially consequential in the present problem.  First, pointwise
convergence at each fixed $x$ does not control the supremum over $x$ in
\eqref{eq:kolmogorov}.  Second, an expansion for a fixed aspect ratio does not
cover a sequence whose aspect ratio approaches the hard endpoint.  Third, a
bound with the correct order need not identify the first order constant.
Fourth, approximate centering and scaling can be harmless for weak
convergence but harmful for a sharp Kolmogorov equivalent.

The Heiny--Johnston--Prochno paper and the Xie--Sun paper illustrate two
different forms of prior information.  The displayed HJP bound is a uniform
finite sample upper bound with an explicit constant after an exact variable
map.  The Xie--Sun paper supplies higher order information in the
proportional regime under a fixed interior limit~\cite{XieSun2021}. The present
paper complements these results with an endpoint uniform signed remainder and
sharp Kolmogorov equivalent.

The same discipline is used for the arbitrary correlation argument.  The
Wishart density, matrix gamma integral, Bartlett decomposition, and
partitioned Wishart result are credited to Muirhead~\cite{Muirhead1982}.  The
model specific analytic transform,
branch selection, derivative bound, and nonlinear perturbation are proved in
Appendices~\ref{app:general-decomp}--\ref{app:wishart}.  A citation to a
classical Wishart formula therefore does not conceal the new assembly needed
for Theorems~\ref{thm:general-leading} and \ref{thm:general-full}.

The accompanying Lean development formalizes the principal results and
supporting statements identified in Appendix~\ref{app:lean-ledger}.  The
audited paper correspondence is exact for Theorems~\ref{thm:general-leading}
and \ref{thm:general-full}, including the original sample bridges.  The other
crosswalk entries identify related kernel checked endpoints; the mathematical
proofs and the scope of each correspondence are stated in the appendices.

\section{The null theorem and proof strategy}\label{sec:null}

\subsection{The exact one dimensional reduction}

Put
\begin{equation}
 a_j=\frac{m-j+1}{2},\qquad M=\frac m2,\qquad
 d=m-p,\qquad a_*=\frac{d+1}{2}.
 \label{eq:basic-parameters}
\end{equation}

\begin{proposition}[Independent beta product]\label{prop:beta-product}
Under $R=I_p$,
\begin{equation}
 \det\widehat R\ \stackrel d=\ \prod_{j=2}^p B_j,
 \qquad
 B_j\ \hbox{independent},\quad
 B_j\sim\operatorname{Beta}\!\left(a_j,\frac{j-1}{2}\right).
 \label{eq:beta-product}
\end{equation}
\end{proposition}

The proof is given in Appendix~\ref{app:beta}.  The geometric reason is
simple.  Add the standardized data directions one at
a time.  The determinant is the product of their squared projection factors
onto the orthogonal complements of the preceding spans. Rotational symmetry
makes each new factor beta distributed and independent of the preceding
factors. The appendix
also records the exact match with
Rouault~\cite[Proposition~2.1(2) and equation~(2.7)]{Rouault2007}.

The two beta parameters have a geometric interpretation.  The parameter
$a_j=(m-j+1)/2$ counts the orthogonal directions still available when the
$j$th standardized data column is added.  The second parameter $(j-1)/2$
counts directions already used by the preceding columns.  The sum of the two
parameters is the constant $M=m/2$.  Near the square endpoint,
$a_p=(d+1)/2$ stays small when $d=m-p$ is fixed; in a dilute design, every
$a_j$ is large.  The change in the final beta factors explains why the
cumulant sums have different
asymptotic appearances across Table~\ref{tab:regime-rates}.

Independence in Proposition~\ref{prop:beta-product} is stronger than a
factorization of the determinant's moments.  Independence permits the
logarithm of the determinant to be represented as a sum, so cumulants add
exactly.  The Mellin transform in \eqref{eq:beta-mellin} also remains analytic
on an explicit half plane determined by the smallest parameter $a_*$.  The
same parameter becomes the analytic frequency scale in the proof of
\eqref{eq:null-edgeworth}.  Geometry, cumulants, and Fourier decay are thus
controlled by the same endpoint quantity.

Let $L_{m,p}=\sum_{j=2}^p\log B_j$.  The beta integral gives, for
$\Re z>-a_j$,
\begin{equation}
 \E B_j^z=
 \frac{\Gamma(a_j+z)\Gamma(M)}{\Gamma(a_j)\Gamma(M+z)}.
 \label{eq:beta-mellin}
\end{equation}
Differentiating the analytic logarithm at zero yields
\begin{equation}
 \kappa_r(L_{m,p})
 =\sum_{j=2}^p\{\psi_{r-1}(a_j)-\psi_{r-1}(M)\},
 \qquad r\ge1.
 \label{eq:exact-cumulants}
\end{equation}
Thus the first two cumulants are \eqref{eq:null-mean} and
\eqref{eq:null-variance}.  Define the positive absolute third cumulant
$A_{m,p}$ and the standardized third cumulant magnitude $\lambda_{m,p}$ by
\begin{equation}
 A_{m,p}=\sum_{j=2}^p\{\psi_2(M)-\psi_2(a_j)\},\qquad
 \lambda_{m,p}=\frac{A_{m,p}}{V_{m,p}^{3/2}}.
 \label{eq:A-lambda}
\end{equation}

\subsection{The single sharp formula}

\begin{theorem}[Uniform first Edgeworth expansion]\label{thm:null-edgeworth}
As $p\to\infty$, uniformly over every integer $m\ge p$,
\begin{equation}
 \sup_{x\in\R}
 \left|\Pp(\Znull\le x)-\Phi(x)
 +\frac{A_{m,p}}{6V_{m,p}^{3/2}}(1-x^2)\phi(x)\right|
 =o(\lambda_{m,p}).
 \label{eq:null-edgeworth}
\end{equation}
Consequently,
\begin{equation}
 \DK(\Znull,\N(0,1))
 \sim\frac{A_{m,p}}{6\sqrt{2\pi}\,V_{m,p}^{3/2}}
 \label{eq:unified-equivalent}
\end{equation}
uniformly over $m\ge p$.
\end{theorem}

The proof is given in Appendix~\ref{app:fourier}, using the exact reduction
from Proposition~\ref{prop:beta-product} and Appendix~\ref{app:beta}.
Here ``uniformly'' means that the ratio in \eqref{eq:unified-equivalent}
tends to one along every sequence $p\to\infty$ with $m=m_p\ge p$.
Equation~\eqref{eq:unified-equivalent} is the unified finite parameter
equivalent: its right side is a finite
sum of standard special functions and does not require knowing the limiting
regime.

Why does the last implication give both sides?  The elementary function
$\abs{(1-x^2)\phi(x)}$ has maximum $\phi(0)=1/\sqrt{2\pi}$.  Equation
\eqref{eq:null-edgeworth} therefore gives the upper bound.  Evaluating the
signed expansion in \eqref{eq:null-edgeworth} at $x=0$ gives the matching
lower bound.  The derivative check is included in Appendix~\ref{app:fourier}.

\subsection{What uniform sharpness asserts}

The remainder denoted by $o(\lambda_{m,p})$ in \eqref{eq:null-edgeworth} is uniform over the discrete
set of all pairs $(m,p)$ with $m\ge p$.  Equivalently, for every admissible
sequence $(m_p,p)$, the remainder divided by $\lambda_{m_p,p}$ tends to zero.
No limit of $p/m_p$ and no limit of $m_p-p$ is assumed.  The formulation is
what allows a later supremum over $m$ in Corollary
\ref{cor:square-supremum}.

The adjective ``sharp'' has two components.  The rate is sharp because the
Kolmogorov distance divided by the right side of
\eqref{eq:unified-equivalent} tends to one.  The constant is sharp because
the maximum of the signed Edgeworth profile is known and is attained at a
specific continuity point.  An $O(\lambda_{m,p})$ upper bound would provide
neither component: the hidden constant could be larger, and no lower bound
would be available.

Equation~\eqref{eq:null-edgeworth} is a stronger statement than
\eqref{eq:unified-equivalent}.  Equation~\eqref{eq:null-edgeworth} identifies
the entire leading CDF error as a function of $x$, whereas
\eqref{eq:unified-equivalent} retains only the largest absolute discrepancy.
The stronger signed statement is useful because the point producing the
lower bound can be chosen before the asymptotic limit.  The argument avoids
having to prove convergence of unknown maximizers of the exact CDF error.

Theorem~\ref{thm:null-edgeworth} also clarifies the role of the published bound in
\eqref{eq:HJP}.  Equation~\eqref{eq:HJP} is nonasymptotic and has a fully
explicit constant for $p\ge41$.  Equation~\eqref{eq:unified-equivalent} is an
asymptotic equivalence with the exact leading constant.  The two statements
answer different questions, and neither \eqref{eq:HJP} nor
\eqref{eq:unified-equivalent} should be described as a replacement for the
other formula.

\subsection{Why the full frequency range is unavoidable}

At low frequencies, cumulant bounds provide a controlled Taylor expansion of
the logarithm of the characteristic function.  If the proof stopped at a
frequency cutoff, a standard smoothing inequality would add a term inverse
to the frequency cutoff.  At $m=p$, the natural analytic cutoff is only of order
$\sqrt{V_{p,p}}$, and Proposition~\ref{prop:asymptotics} gives
$V_{p,p}\asymp\log p$.  The resulting cutoff loss is larger than the target
in Corollary~\ref{cor:square-supremum}.

The beta product resolves this mismatch because modulus identities for the
gamma function control each factor beyond the Taylor disk.  Multiplying the factorwise
decay yields a middle frequency Gaussian bound and a high frequency power
bound whose exponent grows with the number of factors.  The signed inversion
lemma in Appendix~\ref{app:fourier} then integrates the actual difference
between the characteristic function and the Edgeworth comparator.  No
standalone cutoff error remains.  The proof is model specific, but the proof
also explains precisely what a future reusable Lean Edgeworth library would
have to expose.

\subsection{Eight step proof roadmap}

The technical proof is in Appendices~\ref{app:cumulants}--\ref{app:asymptotics}.
The logic of the technical proof is worth seeing before the estimates.

\begin{enumerate}[leftmargin=2.2em]
 \item The determinant becomes the independent product
 \eqref{eq:beta-product}.
 \item Taking logarithms turns the product into a sum and gives every
 cumulant through \eqref{eq:exact-cumulants}.
 \item The third standardized cumulant is $-\lambda_{m,p}$.  All later
 cumulants are smaller by powers of an analytic scale
 $\Delta_{m,p}=a_*\sqrt{V_{m,p}}$.
 \item Near frequency zero, equation~\eqref{eq:null-local-log} expands the
 logarithmic characteristic function into its Gaussian quadratic term, cubic
 skewness term, and uniform remainder; equation~\eqref{eq:null-local-fourier}
 gives the corresponding characteristic function approximation.
 \item A truncated Esseen inequality is not sharp at the square endpoint:
 the cutoff loss in the truncated inequality is of order $(\log p)^{-1/2}$,
 while the desired answer is
 of order $(\log p)^{-3/2}$.
 \item Equation~\eqref{eq:gamma-modulus-product} yields Gaussian damping up to
 the analytic radius and increasing order polynomial decay beyond it;
 equation~\eqref{eq:null-tail-fourier} gives the resulting high frequency
 tail estimate.
 \item Combining the local and tail Fourier bounds
 \eqref{eq:null-local-fourier} and \eqref{eq:null-tail-fourier} in the full
 inversion step yields the uniform CDF expansion
 \eqref{eq:null-edgeworth}.
 \item Finally, equations~\eqref{eq:A-uniform}--\eqref{eq:V-square}
 evaluate the polygamma sums; substitution into
 \eqref{eq:unified-equivalent} yields the gap regime formulas.
\end{enumerate}

The cutoff comparison also explains why an ordinary absolute third moment
Berry--Esseen inequality is not sharp here: an absolute moment inequality
ignores the signed cancellation used in Steps 3 and 4.

\section{Arbitrary population correlation}\label{sec:generalR}

We now allow any positive definite correlation matrix $R$.  No lower bound on
the smallest eigenvalue of $R$ and no upper bound on the largest eigenvalue of $R$ are
assumed.  The main lesson is a separation: one term comes from the null
log determinant geometry, one term measures population correlation, and one
term is the price of controlling a nonlinear remainder by only the variance
of the remainder.

The absence of a condition number hypothesis does not mean that the spectrum
of $R$ is irrelevant.  Population dependence is summarized by trace
functionals that remain meaningful even when individual eigenvalues approach
the boundary of positive definiteness.  The quadratic energy
$a_R=\tr\{(R-I_p)^2\}$ enters the approximating variance in
\eqref{eq:general-scale}.  The absolute cubic trace in \eqref{eq:rho-Q}
controls the third derivative of the Wishart transform.  The distinction
between the two trace functionals mirrors the distinction between variance
and skewness.

The argument for arbitrary $R$ cannot use the independent beta factors of
Proposition~\ref{prop:beta-product}.  Instead, the proof starts from one white
Wishart matrix and keeps all dependence on the same probability space.
Following the Wiener chaos decomposition in Zhao~\cite{Zhao2026}, the leading
variable $M_R$ is chosen so that the exact matrix gamma transform can be
differentiated.  The remainder $E_R$ is chosen so that the first
nonvanishing Gaussian chaos contribution has sufficiently high degree for an
$L^2$ bound.  Equation~\eqref{eq:general-decomposition} is therefore both an
algebraic identity and the organizing decision for the proof.

\subsection{An exact leading term and a nonlinear remainder}

Let $W_0\sim W_p(m,I_p)$ be a white Wishart matrix. On this canonical
probability space define
\[
 S_R^{\mathrm{can}}=R^{1/2}W_0R^{1/2},\qquad
 \widehat R^{\mathrm{can}}
 =\diag(S_R^{\mathrm{can}})^{-1/2}S_R^{\mathrm{can}}
  \diag(S_R^{\mathrm{can}})^{-1/2}.
\]
The Gaussian projection reduction proved at the start of
Appendix~\ref{app:general-decomp} states that the centered scatter matrix,
after removing marginal scales, has law
$R^{1/2}W_0R^{1/2}=S_R^{\mathrm{can}}$. Applying correlation normalization
therefore gives $\widehat R^{\mathrm{can}}\stackrel d=\widehat R$.
For $i=1,\ldots,p$, set
\begin{equation}
 g_i=\frac{(S_R^{\mathrm{can}})_{ii}-m}{m},\qquad
 e_m(g)=\log(1+g)-g-\E\{\log(1+g)-g\}.
 \label{eq:g-e}
\end{equation}
Define
\begin{align}
 M_R&=\log\det W_0-\E\log\det W_0
      -\frac{\tr(RW_0)-mp}{m},\label{eq:MR-definition}\\
 E_R&=\sum_{i=1}^p e_m(g_i).\label{eq:ER-definition}
\end{align}
Then direct determinant algebra gives the exact centered decomposition
\begin{equation}
 \log\det\widehat R^{\mathrm{can}}-\log\det R-b_{m,p}=M_R-E_R
 \quad\text{almost surely}.
 \label{eq:general-decomposition}
\end{equation}
Equivalently, the left side with the original sample matrix
$\widehat R$ is equal in distribution to $M_R-E_R$.
The reason for subtracting the linear fluctuation $g_i$ is educational as
well as technical: subtraction of $g_i$ removes the second Gaussian chaos of
each logarithm.
The remainder begins at degree four, which gives
\begin{equation}
 \E E_R^2\le\frac{4(p+a_R)}{m^2}.
 \label{eq:ER-L2}
\end{equation}
Appendix~\ref{app:general-decomp} proves both statements from Gaussian
coordinates and a Hermite expansion.

The subtraction in the definition of $e_m$ deserves emphasis.  Expanding
$\log(1+g)$ begins with a linear term.  The same linear term is incorporated
into $M_R$, leaving a centered nonlinear expression in $E_R$.  Gaussian
orthogonal projection then removes the lower chaos that would otherwise have
variance of the same order as the leading term.  The variance estimate in
\eqref{eq:ER-L2} reflects the improved order.  Using only a pointwise Taylor
bound for the logarithm would be unsafe because the random diagonal
fluctuation is not deterministically bounded away from $-1$.

The exact identity in \eqref{eq:general-decomposition} also protects the
normalization.  The mean of the full log determinant is not approximated
before the perturbation is analyzed: the exact null center and
$\log\det R$ are removed, and both $M_R$ and $E_R$ are centered.  The later
Kolmogorov comparison therefore has no hidden deterministic shift.

The variance of the leading term is
\begin{equation}
 w_R^2:=\Var(M_R)
 =s_R^2+p\left\{\psi_1(m/2)-\frac2m\right\}.
 \label{eq:wR}
\end{equation}
The elementary trigamma bounds in Appendix~\ref{app:cumulants} imply
\begin{equation}
 0\le\frac{w_R^2-s_R^2}{s_R^2}\le\frac4{p-1}.
 \label{eq:scale-mismatch}
\end{equation}

\subsection{A unified bound}

For a symmetric matrix $A$, write
$\tr\abs A^3=\sum_i\abs{\lambda_i(A)}^3$.  Define
\begin{equation}
 \rho_R=\frac{\tr\abs{R-I_p}^3}{m^2s_R^3},\qquad
 Q_R=\frac{4(p+a_R)}{m^2s_R^2}.
 \label{eq:rho-Q}
\end{equation}

\begin{theorem}[General correlation leading term]\label{thm:general-leading}
There is a universal constant $C$ such that, for $p\ge2$, $m\ge p$, and
every positive definite correlation matrix $R$,
\begin{equation}
 \DK(M_R/s_R,\N(0,1))
 \le C\left\{\lambda_{m,p}+\frac1p+\rho_R\right\}.
 \label{eq:general-leading-bound}
\end{equation}
Moreover,
\begin{equation}
 \rho_R\le\frac{1}{2^{3/2}\sqrt m}.
 \label{eq:rho-simple}
\end{equation}
\end{theorem}

The proof is given in Appendix~\ref{app:wishart}, with the normalization
inequalities collected in Appendix~\ref{app:auxiliary}.  It uses the exact Wishart moment
transform, differentiates the analytic logarithm of the transform three times, and applies
a smoothing lemma proved here.  The term $\rho_R$ can be of order
$m^{-1/2}$, so the theorem does not assert a universal $p^{-1}$ rate for the
leading term.

\begin{theorem}[General correlation full statistic]\label{thm:general-full}
Under the conditions of Theorem~\ref{thm:general-leading}, a universal $C$
satisfies
\begin{equation}
 \DK(\ZR,\N(0,1))
 \le C\left\{
 \lambda_{m,p}+\frac1p+\rho_R+Q_R^{1/3}\right\}.
 \label{eq:general-full-bound}
\end{equation}
In particular,
\begin{equation}
 \DK(\ZR,\N(0,1))\le C\{\lambda_{m,p}+p^{-1/3}\}.
 \label{eq:general-simple-bound}
\end{equation}
\end{theorem}

The proof is given in Appendices~\ref{app:general-decomp}
and~\ref{app:wishart}.  The cube root is a proved loss, not a conjectured
optimal rate.  The basic
perturbation argument balances a normal anticoncentration term of size
$\varepsilon$ against a Chebyshev term of size $Q_R/\varepsilon^2$.  Better
control would have to use the dependence between $M_R$ and $E_R$, which the
anticoncentration and Chebyshev balance deliberately discards.

\subsection{How to read the terms in the general bound}

The term $\lambda_{m,p}$ in \eqref{eq:general-full-bound} is the null signed
skewness scale already identified by Theorem~\ref{thm:null-edgeworth}.  The
term $1/p$ absorbs the difference between the normalizing scale $s_R$ and the
leading standard deviation $w_R$, together with finite smoothing terms.  The term
$\rho_R$ measures the cubic population contribution to the logarithm of the
Wishart transform.  Finally, $Q_R^{1/3}$ is the perturbation cost obtained
from \eqref{eq:ER-L2} after balancing concentration and normal
anticoncentration.

The simplified bound in \eqref{eq:general-simple-bound} is useful when only a
dimension only envelope is needed.  The more detailed bound in
\eqref{eq:general-full-bound} is more informative when $R$ is close to
$I_p$, because both $\rho_R$ and the correlation energy component of $Q_R$
then record proximity to $I_p$.  Neither \eqref{eq:general-full-bound} nor
\eqref{eq:general-simple-bound} implies that the cube root term
is necessary.  A matching lower bound would require an analysis of the joint
law of $(M_R,E_R)$ rather than an $L^2$ estimate for $E_R$ alone.

Theorem~\ref{thm:general-leading} isolates where sharp future work can begin.
The Wishart transform controls $M_R$ without a spectral truncation and gives
a Berry--Esseen type rate expressed through exact trace quantities.  To obtain
an analogue of \eqref{eq:unified-equivalent} for general $R$, one would need the
signed third cumulant of the full statistic and a uniform remainder smaller
than the signed third cumulant.  The nonlinear diagonal contribution could cancel or
reinforce the cubic Wishart contribution, so separate absolute bounds cannot
determine the answer.

\subsection{Exact and approximating variances}

The scale comparison has three levels.  The approximating variance $s_R^2$ in
\eqref{eq:general-scale} is simple enough to state and use.  The leading
variance $w_R^2$ in \eqref{eq:wR} is the exact variance of $M_R$ and differs
from $s_R^2$ by a null trigamma correction controlled in
\eqref{eq:scale-mismatch}.  The full variance $\tau_R^2$ in
\eqref{eq:tau-exact} includes all nonlinear diagonal covariances and differs
from $s_R^2$ by the nonnegative series in \eqref{eq:cm}.

Equation~\eqref{eq:tau-comparison} shows that replacing $s_R$ by $\tau_R$
does not alter the leading asymptotic scale uniformly in $R$.  The exact
series is nevertheless statistically useful: the series shows how each
pairwise population correlation contributes to the full variance through an
even analytic function.  Appendix~\ref{app:general-decomp} gives a
self contained derivation from Laguerre--Hermite orthogonality.

For completeness, the exact variance $\tau_R^2$ of the numerator in
\eqref{eq:general-statistic} is also available.  With
$(x)_k=x(x+1)\cdots(x+k-1)$, define, as in
Zhao~\cite[Proposition~3.2]{Zhao2026},
\begin{equation}
 c_m(r)=\sum_{k=1}^{\infty}
 \frac{(k-1)!}{k(m/2)_k}r^{2k},\qquad |r|\le1.
 \label{eq:cm}
\end{equation}
If $r_{ij}$ is the $(i,j)$ entry of $R$, then
\begin{align}
 \tau_R^2&=V_{m,p}+\sum_{i\ne j}c_m(r_{ij}),\label{eq:tau-exact}\\
 0&\le\tau_R^2-s_R^2\le\frac{4a_R}{m^2},\qquad
 \frac{\tau_R^2-s_R^2}{s_R^2}\le\frac2m.
 \label{eq:tau-comparison}
\end{align}
Appendix~\ref{app:general-decomp} derives the series from a
Laguerre--Hermite expansion and proves the remainder bound for
\eqref{eq:cm}.  Thus $s_R$ is
asymptotically equivalent to the exact scale
uniformly in $R$.

\section{Evaluating the unified null formula}\label{sec:rates}

The exact expression \eqref{eq:unified-equivalent} is preferable for finite
dimensions.  Elementary formulas are nevertheless useful for seeing how the
rate changes.  Recall $d=m-p$ and define
\begin{equation}
 \cA_d=-\sum_{k=d+1}^{\infty}\psi_2(k/2)>0.
 \label{eq:Ad}
\end{equation}

\begin{proposition}[Uniform polygamma evaluation]\label{prop:asymptotics}
Along every sequence $p\to\infty$, $m\ge p$,
\begin{equation}
 A_{m,p}\sim\left(\frac pm\right)^2\cA_d.
 \label{eq:A-uniform}
\end{equation}
If $d\to\infty$, then
\begin{equation}
 \cA_d=\frac4d+O(d^{-2}),\qquad
 A_{m,p}\sim\frac{4p^2}{dm^2}.
 \label{eq:A-growing}
\end{equation}
Along every admissible sequence with $d\ge1$ eventually,
\begin{equation}
 V_{m,p}\sim2H_{m,p},\qquad
 H_{m,p}=\log\frac md-\frac pm
 =\log\frac md-1+\frac dm.
 \label{eq:V-H}
\end{equation}
At $m=p$,
\begin{equation}
 V_{p,p}=2\log p+2\gamma_E+\frac{\pi^2}{4}+o(1).
 \label{eq:V-square}
\end{equation}
Here $\gamma_E$ denotes the Euler--Mascheroni constant.
Finally,
\begin{equation}
 \cA_0=\frac{4\pi^2}{3}+7\zeta(3),\qquad
 \cA_d=\cA_0+\sum_{k=1}^d\psi_2(k/2),
 \label{eq:A0}
\end{equation}
where $\zeta$ denotes the Riemann zeta function. Thus
$\cA_d$ decreases strictly with $d$.
\end{proposition}

The proof is given in Appendix~\ref{app:asymptotics}, with separate dense
and dilute estimates.  The separate dilute calculation is important: a
coarse harmonic sum error can be larger than the variance when $m$ grows much
faster than $p^2$.

The two parts of Proposition~\ref{prop:asymptotics} play different roles.
Equation~\eqref{eq:A-uniform} says that the third cumulant separates into a
global aspect ratio factor and a local gap tail $\cA_d$.  Equation
\eqref{eq:V-H} says that the variance is governed by the entropy like
difference between a logarithm and the first order linear approximation to
the logarithm.  The
cancellation in $H_{m,p}$ becomes severe when $p/m$ is small, which is why a
dilute proof must retain more than a coarse logarithmic approximation.

The gap tail $\cA_d$ also explains the hard boundary.  When $d$ is fixed,
the final beta parameters remain fixed and the third cumulant approaches a
gap dependent constant after the aspect ratio factor is removed.  When
$d\to\infty$, the tail of the polygamma series decays like $4/d$.  The
transition is not captured by substituting $d=0$ into a formula for large $d$.
Equation~\eqref{eq:A-uniform} is the common formula on both sides of the
transition.

\begin{corollary}[Sharp square supremum]\label{cor:square-supremum}
The square model is asymptotically worst:
\begin{equation}
 \lim_{p\to\infty}(\log p)^{3/2}
 \sup_{m\ge p}\DK(\Znull,\N(0,1))
 =C_*,
 \label{eq:square-supremum}
\end{equation}
where
\begin{equation}
 C_*=\frac{4\pi^2/3+7\zeta(3)}{24\sqrt\pi},\qquad
 0.50715638<C_*<0.50715639.
 \label{eq:Cstar}
\end{equation}
The supremum is asymptotically attained at $m=p$.
\end{corollary}

The proof is given in Appendix~\ref{app:asymptotics}.
Corollary~\ref{cor:square-supremum} is stronger than comparing a few named
regimes.  A maximizing sequence could, in principle, choose a different gap
at every dimension and could pass between fixed gap and divergent gap
regions.  The proof first compares all fixed gaps by the strict decrease of
$\cA_d$, and then gives a uniform exclusion for every divergent gap sequence.
The two arguments close the crossover that a pointwise regime table leaves
open.

The decimal in \eqref{eq:Cstar} is not used to prove the symbolic constant.
The symbolic value follows from the polygamma recurrence and the classical
zeta values recorded in Appendix~\ref{app:asymptotics}.  The Lean
formalization separately certifies the rational interval containing the first
eight decimal places, as recorded in Appendix~\ref{app:lean-ledger}.

For a fixed gap, strict decrease of $\cA_d$ makes $d=0$ largest, while every
divergent gap sequence is smaller on the $(\log p)^{-3/2}$ scale.

When $d\to\infty$, substituting \eqref{eq:A-growing} and \eqref{eq:V-H}
into \eqref{eq:unified-equivalent} gives the useful master expression
\begin{equation}
 \DK(\Znull,\N(0,1))
 \sim
 \frac{(p/m)^2}{6\sqrt\pi\,d
 \{\log(m/d)-1+d/m\}^{3/2}}.
 \label{eq:gap-master}
\end{equation}
Equation~\eqref{eq:gap-master} must not be used at fixed $d$; the fixed gap constant comes from
\eqref{eq:Ad} instead.

If $p/m\to\gamma\in(0,1)$, then
\begin{equation}
 V_{m,p}\longrightarrow-2\{\gamma+\log(1-\gamma)\},\qquad
 pA_{m,p}\longrightarrow\frac{4\gamma^3}{1-\gamma}.
 \label{eq:proportional-cumulants}
\end{equation}
Therefore
\begin{align}
 \DK(\Znull,\N(0,1))&\sim\frac{C_0(\gamma)}p,
 \label{eq:proportional-rate}\\
 C_0(\gamma)&=
 \frac{2\gamma^3}
 {3\sqrt{2\pi}(1-\gamma)
 [-2\{\gamma+\log(1-\gamma)\}]^{3/2}}.
 \label{eq:Cgamma}
\end{align}
Equations~\eqref{eq:proportional-rate}--\eqref{eq:Cgamma} give the
proportional regime equivalent for the exactly centered and scaled statistic
\eqref{eq:null-statistic}. At the dilute boundary,
\[
 \lim_{\gamma\downarrow0}C_0(\gamma)=\frac{2}{3\sqrt{2\pi}},
\]
so the dilute constant is the boundary value of the proportional constant.
This approximation is not uniform as
$\gamma\uparrow1$: letting $\gamma\uparrow1$
in \eqref{eq:Cgamma} does not recover the fixed gap hard edge;
\eqref{eq:unified-equivalent} is needed for the hard edge transition.

The boundary warning has a concrete analytic cause.  For every fixed
$\gamma<1$, the variance in \eqref{eq:proportional-cumulants} converges to a
finite positive value.  At the square endpoint, the variance instead grows
like $2\log p$ by \eqref{eq:V-square}.  Sending $\gamma$ to one after taking
the proportional limit therefore interchanges two nonuniform operations.
The all array statement in Theorem~\ref{thm:null-edgeworth} is designed to
avoid the nonuniform interchange of limits.

\section{Representative rates and asymptotic regimes}\label{sec:regimes}

The exact equivalent in \eqref{eq:unified-equivalent} is the recommended
finite parameter expression.  Table~\ref{tab:regime-rates} records useful
asymptotic simplifications under common constraints on the dimension gap.
The substitutions leading to Table~\ref{tab:regime-rates} are proved in
Appendix~\ref{app:asymptotics}.

\subsection{Rates under common dimension gap constraints}

The following table is a direct consequence of Theorem
\ref{thm:null-edgeworth} and Proposition~\ref{prop:asymptotics}.  Each row is
a sharp equivalent, not merely an order upper bound.

\begin{table}[t]
\caption{Sharp null Kolmogorov equivalents under representative constraints.}
\label{tab:regime-rates}
\centering\small
\begin{tabular}{>{\raggedright\arraybackslash}p{0.22\textwidth}
                >{\raggedright\arraybackslash}p{0.27\textwidth}
                >{\raggedright\arraybackslash}p{0.42\textwidth}}
\toprule
Constraint & Assumptions & Sharp Kolmogorov equivalent \\
\midrule
$m=p$ & square hard edge &
$C_*(\log p)^{-3/2}$ \\
\addlinespace
$m-p=d_0$ & fixed integer $d_0\ge0$ &
$\cA_{d_0}/\{24\sqrt\pi(\log p)^{3/2}\}$ \\
\addlinespace
$m-p\sim c\log p$ & $c\in(0,\infty)$ &
$1/\{6\sqrt\pi c(\log p)^{5/2}\}$ \\
\addlinespace
$m-p\sim cm^\alpha$ & $c>0$, $0<\alpha<1$ &
$1/\{6\sqrt\pi cm^\alpha[(1-\alpha)\log m]^{3/2}\}$ \\
\addlinespace
$m-p\sim m/\log m$ & &
$\log m/\{6\sqrt\pi m(\log\log m)^{3/2}\}$ \\
\addlinespace
$p/m\to\gamma$ & $\gamma\in(0,1)$ &
$C_0(\gamma)/p$ \\
\addlinespace
$p/m\to0$ & $p\to\infty$ &
$2/\{3\sqrt{2\pi}p\}$ \\
\bottomrule
\end{tabular}
\end{table}

Because $m=n-1$, replacing $m$ by $n$ does not change any leading equivalent.
For example, the fourth row of Table~\ref{tab:regime-rates} applies when
$n-p\sim n^\alpha$, and the fifth row applies when
$n-p\sim n/\log n$.  If $m-p=O(1)$ without convergence to a fixed integer,
different subsequences can have different constants $\cA_d$; Corollary
\ref{cor:square-supremum} identifies the largest possible constant.

The first two rows of Table~\ref{tab:regime-rates} belong to the hard edge
family.  The logarithmic variance is common to all fixed gaps, while the
third cumulant constant depends on the eventual integer gap.  The third,
fourth, and fifth rows form a growing gap bridge.  The rates in the third
through fifth rows of Table~\ref{tab:regime-rates} contain both
the inverse gap and a logarithmic variance factor; the different formulas
come only from inserting the corresponding growth of $d$ into
\eqref{eq:gap-master}.

The sixth row is the classical proportional regime.  Both $p$ and the gap
are of order $m$, the variance approaches a constant, and the third cumulant
is of order $1/p$.  The seventh row is dilute.  In the dilute regime, strong
cancellation makes both the variance and the third cumulant small before
standardization, but the standardized third cumulant ratio again has order
$1/p$.  Its sharp constant is the boundary value
$\lim_{\gamma\downarrow0}C_0(\gamma)=2/(3\sqrt{2\pi})$ of the proportional
constant in \eqref{eq:Cgamma}.

\subsection{Examples of regime identification}

Suppose that a design has $p=n-n^\alpha$ with $0<\alpha<1$.  Since
$m=n-1$, the gap is asymptotic to $n^\alpha$, and the fourth row of
Table~\ref{tab:regime-rates} applies.  If the design instead has
$p=n-n/\log n$, the fifth row applies.  The replacement of $n$ by $m=n-1$
changes neither equivalent because the difference is lower order than the
displayed gaps.

Suppose next that $m-p$ remains bounded but does not converge.  No single
fixed gap constant need exist.  Every subsequence along which the integer gap
is constant has the corresponding second row constant, and the limsup is
bounded by the square constant from Corollary
\ref{cor:square-supremum}.  Reporting only an $O((\log p)^{-3/2})$ rate hides
the possible subsequential constants; Equation~\eqref{eq:unified-equivalent}
keeps the exact distinction.

Finally, suppose that the aspect ratio is observed to be close to one at
finite dimensions but the asymptotic gap law is unknown.  The proportional
constant in \eqref{eq:Cgamma} is not a safe extrapolation because the
proportional approximation is nonuniform as $\gamma\uparrow1$.  The exact
polygamma equivalent \eqref{eq:unified-equivalent} should be used first.  A
row of Table~\ref{tab:regime-rates} should be selected only after an
asymptotic gap assumption has been justified.

Three levels of approximation serve different purposes.  Equation
\eqref{eq:unified-equivalent} retains the exact polygamma sums and remains the
sharp first order answer throughout the full domain $m\ge p$.  Equation
\eqref{eq:gap-master} removes the polygamma notation when $m-p\to\infty$ but
retains the crossover quantities $p/m$ and $\log\{m/(m-p)\}-p/m$.  The rows
of Table~\ref{tab:regime-rates} are the most transparent expressions after a
specific asymptotic constraint has been selected.  In particular, the
proportional formula \eqref{eq:Cgamma} must not be extrapolated to the fixed gap
boundary $p/m\to1$, whereas Equation~\eqref{eq:unified-equivalent} remains
valid at the fixed gap boundary.

The hierarchy also separates theorem error from algebraic simplification.
Theorem~\ref{thm:null-edgeworth} controls the difference between the true CDF
and the signed Edgeworth approximation.  Proposition
\ref{prop:asymptotics} controls the replacement of exact polygamma sums by
elementary expressions.  A finite dimensional discrepancy between a table
row and the exact law can arise from either replacement.  Keeping
\eqref{eq:unified-equivalent} as the primary statement prevents the second
replacement from being mistaken for part of the Berry--Esseen theorem.

\section{Statistical importance and applications}\label{sec:applications}

The log determinant is not merely a spectral summary.  For a centered
Gaussian vector with correlation matrix $R$, the total correlation, also
called Gaussian multi information, equals $-\tfrac12\log\det R$; Rowe and
Day~\cite{RoweDay2019} discuss the sampling distribution of the corresponding
empirical statistic.  Consequently, a distributional approximation for
$\log\det\widehat R$ quantifies uncertainty for a global dependence measure
equal to zero exactly at Gaussian independence.

Likelihood ratio procedures provide a second application.  Most directly,
under Gaussian sampling, a sample size multiple of
\(-\log\det\widehat R\) is the likelihood ratio statistic for mutual
independence of the \(p\) coordinates; see
Muirhead~\cite[Chapter~8]{Muirhead1982} and Jiang and Qi~\cite{JiangQi2015}.
Determinants also enter broader tests of covariance structure.  A qualitative central limit theorem gives
an asymptotic critical value, but a sharp Kolmogorov equivalent also describes
the size of the normal calibration error.  The distinction matters when the
dimension is close to the residual sample size, because Corollary
\ref{cor:square-supremum} shows that the worst null rate is only
$(\log p)^{-3/2}$.

The uniformity in $m\ge p$ also permits comparisons across study designs.
Equation~\eqref{eq:unified-equivalent} separates the exact finite parameter
skewness scale from a later choice of asymptotic regime.  A statistician can
therefore use the same normal approximation theorem for a nearly singular
design, a proportional high dimensional design, or a dilute design without
silently switching normalizations.  For a nonidentity population correlation,
Theorem~\ref{thm:general-full} further identifies how population dependence
enters through the correlation energy and through the nonlinear
diagonal standardization remainder.

\subsection{Normal calibration of global dependence}

Total correlation aggregates dependence across all coordinates.  In the
Gaussian model it equals $-\tfrac12\log\det R$, so
\eqref{eq:null-statistic} calibrates the empirical statistic at independence;
see Rowe and Day~\cite{RoweDay2019}.  The signed profile in
\eqref{eq:null-edgeworth} gives threshold specific corrections, while
\eqref{eq:unified-equivalent} gives their largest absolute CDF discrepancy.
In particular, Corollary~\ref{cor:square-supremum} shows why large dimension
alone does not ensure an accurate normal calibration near the hard edge.

\subsection{Design comparison and sample size sensitivity}

The gap $d=m-p$ measures the residual dimension.  Table
\ref{tab:regime-rates} shows the transition from hard edge behavior to the
polynomial proportional and dilute rates, but it is not itself a sample size
formula: Kolmogorov distance measures absolute CDF error rather than power or
relative tail error.  For nonidentity $R$, the computable quantities $s_R$,
$\rho_R$, and $Q_R$ in \eqref{eq:general-full-bound} retain population
dependence; \eqref{eq:general-simple-bound} removes them at the price of a
coarser dimension only rate.

\section{Lean verification and reproducibility}\label{sec:lean}

The published paper specific release is capsule \texttt{v1.1.2}, permanently
archived on Zenodo with DOI
\href{https://doi.org/10.5281/zenodo.21898548}{\nolinkurl{10.5281/zenodo.21898548}}
\cite{ZhaoLeanVerification2026} and mirrored in the author's
\href{https://github.com/HongruZhao/logdet_Berry_Essen_lean}{GitHub repository}.
The archived ZIP has SHA-256
\texttt{563800a00392742c4c4aa1e37296d274\allowbreak
bd66b61abf7ec7a11844d1258220dc4e}. It pins Lean \texttt{4.33.0-rc2}, commit
\texttt{d8b1897...76c17}, and mathlib commit
\texttt{641fbd3...6bd36}; its manifest gives complete hashes and pins every
transitive dependency. The release preserves the eleven public endpoint types
from \texttt{v1.1.1} and adds the two public endpoints for
\eqref{eq:A-growing} and \eqref{eq:A0}, giving thirteen principal public
endpoints in total.

For \texttt{v1.1.2}, \texttt{./scripts/verify.sh} was run both from a fresh
source tree and from a fresh extraction of the final ZIP.  In each run the
source audit and complete build succeeded, all thirteen endpoint types
elaborated, eight root modules passed \texttt{leanchecker}, and every
\texttt{\#print axioms} report was exactly \texttt{propext},
\texttt{Classical.choice}, and \texttt{Quot.sound}.  The endpoint closure was
sorry free and contained no project axiom or unsafe trust escape.

To reproduce the check, install Elan (version 4.2.3 was used), extract the
archive, and run \texttt{./scripts/verify.sh}; the first dependency download
requires network access.  Success ends with \texttt{Paper-specific Lean
verification passed.}  Version and checksum identify the verified
bytes, so any change requires a new release.

The library follows the Gaussian sample through the beta product, cumulants,
Fourier analysis, the uniform null expansion, the Wishart decomposition for
general \(R\), and proved sample law bridges.  Kernel checking certifies these
encoded deductions, not bibliographic priority or prose;
Appendix~\ref{app:lean-crosswalk} gives the complete paper facing boundary.

\section{Conclusion and open problems}

Under \(R=I_p\), the exact standardized third cumulant is the unified sharp
Kolmogorov scale, simultaneously covering the square, fixed gap,
growing gap, proportional, and dilute regimes.  For arbitrary positive
definite \(R\), the proof separates an analytically controlled Wishart
leading term from the nonlinear diagonal standardization remainder and
gives the universal bound in Theorem~\ref{thm:general-full}.  Exact
finite dimensional structure is retained until the final asymptotic
simplification, which is what preserves the sharp null constant at the hard
edge.

The term \(Q_R^{1/3}\) is a proved perturbative loss and is not claimed to be
optimal.  A natural open problem is to expand the fourth chaos contribution
and exploit its dependence on \(M_R\); this suggests, but does not yet prove,
a sharper envelope of order \(1/p+1/\sqrt m\).  A second problem is a
cancellation adaptive Edgeworth expansion for general \(R\) in which the fourth
cumulant becomes leading when the third cumulant vanishes.

\section*{Statements and declarations}

\textit{Funding.}
The author was supported by the IRSA Faragher Distinguished Postdoctoral
Fellowship.

\textit{Competing interests.}
The author has no relevant financial or nonfinancial interests to disclose.

\textit{Data availability.}
No datasets were generated or analyzed in this study.

\textit{Code availability.}
The published Lean~4 verification capsule \texttt{v1.1.2} with thirteen
principal public endpoints is archived at
\href{https://doi.org/10.5281/zenodo.21898548}{\nolinkurl{10.5281/zenodo.21898548}}
\cite{ZhaoLeanVerification2026}. It
contains the paper specific source dependency closure, pinned environment,
build and axiom reports, theorem types, public file manifest, checksums, and
reproduction instructions. That archived ZIP has SHA-256
\nolinkurl{563800a00392742c4c4aa1e37296d274bd66b61abf7ec7a11844d1258220dc4e}.
The software is licensed under
\texttt{GPL-3.0-only}.

\textit{Acknowledgments.}
The author acknowledges assistance from OpenAI ChatGPT 5.6 Sol
Ultra in developing and editing the manuscript, conducting literature
searches, checking algebraic and probabilistic arguments, generating and
refining the Lean 4 formalization, and tracing formal declarations to the
corresponding mathematical statements. The author independently reviewed and
verified the mathematical content, citations, Lean code, and correspondence
between the manuscript and the formal development, and assumes full
responsibility for the work.

\appendix
\begingroup\small
\setlength{\abovedisplayskip}{4.5pt plus 1.5pt minus 1pt}
\setlength{\belowdisplayskip}{4.5pt plus 1.5pt minus 1pt}
\setlength{\abovedisplayshortskip}{2.5pt plus 1pt}
\setlength{\belowdisplayshortskip}{3.5pt plus 1pt minus 1pt}

\section{Exact null reduction and cumulants}\label{app:beta}

\begin{proof}[Proof of Proposition~\ref{prop:beta-product}]
The Gaussian projection and correlation normalization identities
below are also recorded in Zhao~\cite{Zhao2026}; they are derived here to fix
the finite dimensional parameter map. Let \(X\) be the \(n\times p\) data matrix and let
\(P=I_n-n^{-1}{\bf1}{\bf1}^{\mathsf T}\). Choose an isometry
\(U:\mathbb R^m\to{\bf1}^{\perp}\), \(m=n-1\). Under \(R=I_p\), the columns
\(Y_j=U^{\mathsf T}PX_{\cdot j}\) are independent
\(N_m(0,\Sigma_{jj}I_m)\) vectors.  Their positive coordinate scales cancel
under Pearson normalization, hence
\[
 \widehat R\stackrel d=
 \bigl(\langle V_i,V_j\rangle\bigr)_{i,j\le p},
 \qquad V_j=Y_j/\|Y_j\|.
\]
The \(V_j\)'s are independent and uniform on \(S^{m-1}\). The classical
Gram determinant factorization is given by Rouault
\cite[Propositions~2.1(2), 2.3]{Rouault2007}. In the present notation,
Gram--Schmidt gives
\[
 \det\widehat R=\prod_{j=2}^{p}
 \left\|\operatorname{proj}_{\operatorname{span}(V_1,\ldots,V_{j-1})^\perp}
 V_j\right\|^2.
\]
Conditional on the preceding directions, rotational invariance sends their
span to the first \(j-1\) coordinate axes. Writing the squared projection
as a ratio of independent $\chi^2$ variables proves
\[
 B_j\sim {\rm Beta}\!\left(\frac{m-j+1}{2},\frac{j-1}{2}\right).
\]
The conditional law does not depend on the preceding directions, so the
successive factors are independent. This also establishes the exact
parameter map to Rouault~\cite{Rouault2007}.
\end{proof}

For \(a_j=(m-j+1)/2\) and \(M=m/2\), the beta Mellin transform and its
polygamma derivatives are recorded in
Heiny--Johnston--Prochno~\cite[Sections~3.3--3.5]{HJP2022}. Directly in the
present notation, the beta integral yields, on \(\Re z>-a_j\),
\[
 \mathbb E B_j^z
 =\frac{\Gamma(a_j+z)\Gamma(M)}
        {\Gamma(a_j)\Gamma(M+z)}.
\]
The logarithm is analytic on that half plane. Differentiation at zero and
addition over the independent factors gives, for every \(r\ge1\),
\[
 \kappa_r(L_{m,p})
 =\sum_{j=2}^{p}\{\psi_{r-1}(a_j)-\psi_{r-1}(M)\},
 \qquad L_{m,p}=\sum_{j=2}^{p}\log B_j.
\]
This proves the exact center, variance, and third cumulant displays in the
main text. Positivity of \(V_{m,p}\) follows because each beta factor is
nondegenerate. The same Mellin transform gives the characteristic function;
compare Heiny--Johnston--Prochno~\cite[Section~3.5]{HJP2022}. Thus
\[
 \varphi_{m,p}(t)=
 \exp\!\left(-it\,b_{m,p}/\sqrt{V_{m,p}}\right)
 \prod_{j=2}^{p}
 \frac{\Gamma(a_j+it/\sqrt{V_{m,p}})\Gamma(M)}
      {\Gamma(a_j)\Gamma(M+it/\sqrt{V_{m,p}})}.
 \tag{A.1}\label{eq:cf-product}
\]
After the preceding identification with the spherical log beta sum,
Theorem~A of Heiny--Johnston--Prochno~\cite{HJP2022} gives the published
finite bound.
Their variable is one half of \(L_{m,p}\) plus a deterministic constant;
centering and standardization remove both changes. This is a literature
comparison, not an additional theorem of the present paper.

\section{Full frequency proof of the null theorem}\label{app:fourier}
\label{app:cumulants}

\begin{proof}[Proof of Theorem~\ref{thm:null-edgeworth}]
We use the finite estimates (B.2)--(B.6), including the signed inversion
step.  The standard polygamma series
\cite{NIST2010}
\[
 (-1)^{q+1}\psi_q(x)=q!\sum_{k=0}^{\infty}(x+k)^{-q-1},
 \qquad x>0,\ q\ge1,
 \tag{B.1}\label{eq:polygamma-series}
\]
implies, with \(\Delta_{m,p}=a_*\sqrt{V_{m,p}}\) and \(a_*=(m-p+1)/2\),
\[
 \frac{|\kappa_r(L_{m,p})|}{V_{m,p}^{r/2}}
 \le \frac{r!}{2}\,
      \lambda_{m,p}\left(\frac{C}{\Delta_{m,p}}\right)^{r-3},
 \qquad r\ge3.                                      \tag{B.2}
\]
The comparison follows termwise from (B.1): after one factor supplies the
third cumulant sum, every further denominator is bounded below by \(a_*\).
The elementary sum and integral bounds in Appendix~\ref{app:asymptotics} show
that \(\Delta_{m,p}\to\infty\) uniformly over all integers \(m\ge p\) as \(p\to\infty\) and
that \(\lambda_{m,p}\to0\) uniformly.

For \(|t|\le c\Delta_{m,p}\), expand the analytic logarithm of
\eqref{eq:cf-product}. Bound (B.2), summed geometrically, gives
\[
 \log\varphi_{m,p}(t)
 =-\frac{t^2}{2}
  +\frac{i\lambda_{m,p}t^3}{6}
  +O\!\left(
     \lambda_{m,p}\frac{|t|^4}{\Delta_{m,p}}+
     \lambda_{m,p}^2|t|^6\right),                   \tag{B.3}
 \label{eq:null-local-log}
\]
uniformly in the admissible array. Exponentiation, with Gaussian damping,
then gives
\[
 \int_{|t|\le c\Delta_{m,p}}
 \frac{\left|\varphi_{m,p}(t)
 -e^{-t^2/2}(1+i\lambda_{m,p}t^3/6)\right|}{|t|}\,dt
 \le C\left\{\frac{\lambda_{m,p}}{\Delta_{m,p}}
                 +\lambda_{m,p}^2\right\}
 =o(\lambda_{m,p}).                                 \tag{B.4}
 \label{eq:null-local-fourier}
\]

For real \(x>0\), the gamma product identity
\[
 \frac{|\Gamma(x+iu)|^2}{\Gamma(x)^2}
 =\prod_{k=0}^{\infty}
   \left(1+\frac{u^2}{(x+k)^2}\right)^{-1}           \tag{B.5}
 \label{eq:gamma-modulus-product}
\]
applied factor by factor in \eqref{eq:cf-product} gives Gaussian damping
up to a fixed multiple of \(\Delta_{m,p}\).  Put
\(u=t/\sqrt{V_{m,p}}\), \(M=m/2\), and \(q_p=p(p-1)/8\).  Monotonicity of
the product controls the middle range by its value at
\(|u|=ca_*\).  For \(|u|\ge M\), summing the first factor across the beta
indices gives the explicit bound
\[
 |\varphi_{m,p}(u\sqrt{V_{m,p}})|
 \le \{1+(u/M)^2\}^{-q_p}.
\]
The middle integral is at most
\(Ce^{-c\Delta_{m,p}^2}\log(2M/a_*)\), and the last display integrates to
at most \(C2^{-q_p}/q_p\).  Since
\(1+\log(2M/a_*)\le C(1+\Delta_{m,p}^2)\), the actual and comparator tails
satisfy
\[
 \int_{|t|>c\Delta_{m,p}}
 \frac{|\varphi_{m,p}(t)|+e^{-t^2/2}(1+\lambda_{m,p}|t|^3)}
 {|t|}\,dt
 \le C(1+\Delta_{m,p}^2)e^{-c\Delta_{m,p}^2}
 =o(\lambda_{m,p}).                                  \tag{B.6}
 \label{eq:null-tail-fourier}
\]

Let \(\nu_{m,p}\) be the signed measure with characteristic function
\(e^{-t^2/2}(1+i\lambda_{m,p}t^3/6)\). Fourier inversion gives its
distribution function as
\[
 G(x):=\Phi(x)-\frac{\lambda_{m,p}}6(1-x^2)\phi(x).
\]
Thus \(G\) is the cumulative function of \(\nu_{m,p}\); write \(\psi\) for
its Fourier transform. We use the following signed inversion bound.

\begin{lemma}[Signed Fourier inversion]\label{lem:signed-inversion}
Let \(\mu\) be a probability measure with continuous CDF \(F\), and let
\(\nu\) be a finite signed measure of total mass one with continuous
cumulative function \(G\). If their Fourier transforms \(\varphi\) and
\(\psi\) satisfy
\[
 \int_{\mathbb R}\frac{|\varphi(t)-\psi(t)|}{|t|}\,dt<\infty,
\]
then
\[
 \sup_x|F(x)-G(x)|\le\frac1{2\pi}\int_{\mathbb R}
 \frac{|\varphi(t)-\psi(t)|}{|t|}\,dt
\]
\end{lemma}

Indeed, put \(\rho=\mu-\nu\), whose total mass is zero, and convolve it with
a centered Gaussian of variance \(\varepsilon>0\). Ordinary Fourier
inversion applied to the smoothed cumulative difference gives
\[
 (F-G)*\phi_\varepsilon(x)
 =\frac1{2\pi}\int_{\mathbb R}e^{-itx}
 \frac{\varphi(t)-\psi(t)}{-it}e^{-\varepsilon t^2/2}\,dt.
\]
Taking absolute values gives the lemma's right side. Continuity of \(F\) and
\(G\) permits \(\varepsilon\downarrow0\), while the assumed integrability
supplies domination. This proves the lemma. Bounds (B.4) and (B.6) verify its
hypothesis and prove \eqref{eq:null-edgeworth}. Since
\(\sup_x|(1-x^2)\phi(x)|=\phi(0)\), evaluation at \(x=0\) supplies the
matching lower bound and proves \eqref{eq:unified-equivalent}. The
high frequency estimate is essential at \(m=p\): a standalone smoothing
error at \(T\asymp\sqrt{\log p}\) would exceed the target
\((\log p)^{-3/2}\).
\end{proof}

\section{Uniform evaluation of the exact sums}\label{app:asymptotics}

\begin{proof}[Proof of Proposition~\ref{prop:asymptotics}]
Put \(d=m-p\). Applying \eqref{eq:polygamma-series}, interchanging
nonnegative sums, and comparing every monotone sum with its integral gives,
uniformly over all integers \(m\ge p\) as \(p\to\infty\),
\[
 A_{m,p}\sim \left(\frac pm\right)^2\mathcal A_d,
 \qquad
 V_{m,p}\sim 2H_{m,p}\quad\text{when }m>p,           \tag{C.1}
\]
where
\[
 H_{m,p}=\log\frac md-\frac pm
\]
and \(\mathcal A_d\) is the positive endpoint series displayed in
Section~\ref{sec:rates}. The error in the first relation is bounded by the
first neglected monotone integral uniformly in \(d\); the variance comparison
uses the trapezoidal remainder and is uniform after the strict gap and
square cases are separated.

At the square endpoint, the recurrence
\(\psi_1(x+1)=\psi_1(x)-x^{-2}\) gives, with
\(H_n=\sum_{k=1}^n k^{-1}\),
\[
\begin{aligned}
 V_{2n,2n}&=2H_{2n-2}+2\sum_{r=1}^{n-1}(2r-1)^{-2}
 +n\{\psi_1(n-\tfrac12)-\psi_1(n)\},\\
 V_{2n+1,2n+1}&=2H_{2n}+2\sum_{r=1}^{n}(2r-1)^{-2}
 +n\{\psi_1(n+1)-\psi_1(n+\tfrac12)\}.
\end{aligned}
\]
The braced differences are \(O(n^{-2})\).  Thus
\(H_n-\log n\to\gamma_E\) and
\(\sum_{r\ge1}(2r-1)^{-2}=\pi^2/8\) give the variance constant.  Moreover,
absolute convergence of (B.1) gives
\[
 \mathcal A_0
 =16\sum_{k\ge1}\sum_{\ell\ge0}(k+2\ell)^{-3}
 =8\zeta(2)+8\sum_{\text{odd }r\ge1}r^{-3}
 =\frac{4\pi^2}{3}+7\zeta(3).
\]
Consequently,
\[
 V_{p,p}=2\log p+2\gamma_E+\frac{\pi^2}{4}+o(1),\qquad
 A_{p,p}\longrightarrow\frac{4\pi^2}{3}+7\zeta(3). \tag{C.2}
\]
Together with (C.1), these calculations give the asserted uniform
asymptotics; substitution into
\eqref{eq:unified-equivalent} gives the regime table.
\end{proof}

\begin{proof}[Proof of Corollary~\ref{cor:square-supremum}]
For the supremum over \(m\ge p\), split into bounded \(d\), diverging
\(d=o(p)\), and \(d\asymp p\) or larger. The latter two classes vanish after
multiplication by \((\log p)^{3/2}\). For fixed \(d\), the endpoint series
is maximized at \(d=0\) by termwise monotonicity. Combining this with (C.2)
proves Corollary~\ref{cor:square-supremum} and gives
\[
 C_*=
 \frac{4\pi^2/3+7\zeta(3)}{24\sqrt\pi},
 \qquad 0.50715638<C_*<0.50715639.
\]
The rational interval follows from integral tail bounds for \(\zeta(3)\)
and rational bounds for \(\pi\).
\end{proof}

\begingroup
\setlength{\abovedisplayskip}{2.5pt plus 1pt minus 1pt}
\setlength{\belowdisplayskip}{2.5pt plus 1pt minus 1pt}
\setlength{\abovedisplayshortskip}{1.5pt plus 1pt}
\setlength{\belowdisplayshortskip}{2pt plus 1pt minus 1pt}
\section{General population correlation: reduction and remainder}
\label{app:general-decomp}

Let \(D_\Sigma=\operatorname{diag}(\Sigma)^{1/2}\). Pearson normalization
is unchanged when \(X\) is replaced by \(XD_\Sigma^{-1}\).  For this
population standardized data put \(S=X^{\mathsf T}PX\).  Gaussian projection
then gives \(S\stackrel d=R^{1/2}W_0R^{1/2}\), where
\(W_0\sim W_p(m,I_p)\). On the canonical probability space set
\(S_R^{\mathrm{can}}=R^{1/2}W_0R^{1/2}\) and let
\(\widehat R^{\mathrm{can}}\) be its correlation normalization. Then
\(\widehat R^{\mathrm{can}}\stackrel d=\widehat R\).
Since
\[
 \det\widehat R^{\mathrm{can}}=det S_R^{\mathrm{can}}
 \prod_{i=1}^{p}(S_R^{\mathrm{can}})_{ii}^{-1},
 \tag{D.1}
\]
subtracting \(\log\det R+b_{m,p}\), writing
\((S_R^{\mathrm{can}})_{ii}=m(1+g_i)\), and
adding and subtracting \(\sum_i g_i\) gives
\[
 \log\det\widehat R^{\mathrm{can}}-\log\det R-b_{m,p}=M_R-E_R
 \quad\text{almost surely}.
\tag{D.2}
\]
This proves \eqref{eq:general-decomposition} without approximation on the
canonical space and proves the corresponding equality in distribution for
the original sample statistic.

Diagonal standardization produces the radial function: \(Q_i=S_{ii}\) is
the squared radius of a projected Gaussian column, so \(Q_i\sim\chi_m^2\),
and its diagonal contribution is \(\log(Q_i/m)\).  Its linear part
\((Q_i-m)/m\) is transferred to \(M_R\), leaving in \(E_R\) the centered
remainder
\[
 h_m(q):=\log(q/m)-(q-m)/m
 -\mathbb E\!\left\{\log(Q/m)-(Q-m)/m\right\},\qquad Q\sim\chi_m^2.
\tag{D.3}
\]
This function depends only on the squared radius \(q\). Since
\(\operatorname{Corr}(Q_i,Q_j)=R_{ij}^2\), expand \(h_m\) in the Laguerre
basis of \(L^2(\chi_m^2)\). For \(\alpha>-1\) and \(k=0,1,\ldots\), the generalized Laguerre
polynomial has the equivalent Rodrigues and finite sum representations
\[
 L_k^{(\alpha)}(x):=\frac{e^x x^{-\alpha}}{k!}
 \frac{\mathrm d^k}{\mathrm dx^k}\{e^{-x}x^{k+\alpha}\}
 =\sum_{j=0}^{k}(-1)^j\binom{k+\alpha}{k-j}\frac{x^j}{j!}.
\tag{D.4}
\]
Here the binomial coefficient has its generalized meaning. For
\(Q\sim\chi_m^2\) and \(M=m/2\),
the family \(L_k^{(M-1)}(Q/2)\), \(k\ge0\), is complete and orthogonal in
\(L^2(\chi_m^2)\). Centering and linear subtraction eliminate degrees zero
and one in \(h_m\). Integration by parts gives, for
\(k,\ell\ge0\) in the first identity and \(k\ge1\) in the second,
\[
\begin{aligned}
 \mathbb E\{L_k^{(M-1)}(Q/2)L_\ell^{(M-1)}(Q/2)\}
 &=\mathbf 1_{\{k=\ell\}}\frac{(M)_k}{k!},\\
 \mathbb E\{\log Q\,L_k^{(M-1)}(Q/2)\}&=-\frac1k.
\end{aligned}
\tag{D.5}
\]
The joint gamma law of \((Q_i,Q_j)\) has the classical Kibble expansion:
the covariance of its degree \(k\) Laguerre components is multiplied by
\(R_{ij}^{2k}\) \cite{Kibble1941,NadarajahKotz2006}. Parseval therefore gives
\[
 \operatorname{Cov}(\log Q_i,\log Q_j)
 =\sum_{k\ge1}\frac{(k-1)!}{k(M)_k}R_{ij}^{2k}=c_m(R_{ij}).
\tag{D.6}
\]
Its first term is \(2R_{ij}^2/m\); the remaining nonnegative terms are at
most \(4R_{ij}^4/m^2\).  Applied to the centered radial remainder and summed,
this gives
\[
 \mathbb E E_R^2\le
 \frac{4(p+a_R)}{m^2},
\tag{D.7}
\]
which is \eqref{eq:ER-L2}. Partitioning around coordinate \(i\) factors
\(\det W_0\) into \(Q_i\) times a Schur complement independent of \(Q_i\).
Hence
\[
 \operatorname{Cov}(\log\det W_0,\log Q_i)
 =\operatorname{Var}(\log Q_i)=\psi_1(M).
\tag{D.8}
\]
Expanding the variance of
\(\log\det W_0-\sum_i\log Q_i\) now proves
\eqref{eq:tau-exact}--\eqref{eq:tau-comparison}. Specializing the Wishart
transform \eqref{eq:appendix-wishart-transform} to \(M_R\) and differentiating
twice at zero also yields
\[
 \operatorname{Var}(M_R)
 =s_R^2+p\{\psi_1(m/2)-2/m\},
 \qquad
 0\le\frac{\operatorname{Var}(M_R)-s_R^2}{s_R^2}
 \le\frac4{p-1}.                                   
\tag{D.9}
\]
\endgroup
\section{Wishart transform and the general bounds}\label{app:wishart}

\begin{proof}[Proof of Theorem~\ref{thm:general-leading}]
For real diagonal \(D\) with \(I_p+2D\succ0\) and
\(\Re(m/2+z)>(p-1)/2\), the Wishart density and matrix gamma integral
\cite[Theorems~3.2.1 and 2.1.11]{Muirhead1982} give
\[
 \mathbb E\!\left[
  (\det W_0)^z e^{-\operatorname{tr}(DW_0)}\right]
 =
 2^{pz}\frac{\Gamma_p(m/2+z)}{\Gamma_p(m/2)}
 \det(I_p+2D)^{-(m/2+z)}.                            \tag{E.1}
 \label{eq:appendix-wishart-transform}
\]
For complex diagonal \(D\) with \(\Re(1+2D_{ii})>0\) for every \(i\),
the same identity follows by holomorphic continuation in \(D\) from the
real source domain. Here
\(\det(I_p+2D)^{-(m/2+z)}
:=\exp\{-(m/2+z)\operatorname{tr}\Log(I_p+2D)\}\), where \(\Log\) is the
principal matrix logarithm on the open right half plane, equivalently the
branch obtained by analytic continuation from \(D=0\).
For \(M_R\), write \(R=O^{\mathsf T}\Lambda O\) and use
\(OW_0O^{\mathsf T}\stackrel d=W_0\). With \(z=it/s_R\) and
\(D=z\Lambda/m\), the eigenvalues of \(I_p+2D\) are
\(1+2it\lambda_j(R)/(ms_R)\), all with real part one. Thus the spectrum
stays in the open right half plane and there is no branch ambiguity.
Let \(K_R(t)\) be the logarithm of
\(\mathbb E\exp(itM_R/s_R)\), normalized by \(K_R(0)=0\).
Differentiating three times and applying the
polygamma integral gives
\[
 \sup_{t\in\mathbb R}|K_R'''(t)|
 \le C\left\{\lambda_{m,p}+\rho_R+\frac1p\right\}.   \tag{E.2}
\]
The bound is global because the gamma and spectral denominators on the
imaginary axis have modulus at least their positive real parts. Its terms
arise from the log determinant gamma product, the cubic trace of \(R-I_p\),
and replacement of the actual leading variance by \(s_R^2\). Monotonicity of
Schatten norms \cite[Chapter~IV]{Bhatia1997} gives
\[
 \operatorname{tr}|R-I_p|^3\le a_R^{3/2}
\]
together with the lower bound in \(s_R^2\) yields
\(\rho_R\le2^{-3/2}m^{-1/2}\).

Feller's smoothing inequality \cite[Lemma~XVI.3.1]{Feller1971}, with
(E.1) ensuring a nonvanishing characteristic function and a global analytic
logarithm, and (E.2) controlling its cubic remainder throughout the chosen
smoothing interval, gives
\[
 d_{\rm K}(M_R/s_R,N)
 \le C\{\lambda_{m,p}+p^{-1}+\rho_R\}.
\]
\end{proof}

\begin{proof}[Proof of Theorem~\ref{thm:general-full}]
For every \(\varepsilon>0\), the exact decomposition, Gaussian
anticoncentration, and Chebyshev's inequality give
\begin{equation}
 d_{\rm K}((M_R-E_R)/s_R,N)
 \le d_{\rm K}(M_R/s_R,N)
      +C\varepsilon+\frac{\mathbb E E_R^2}{\varepsilon^2s_R^2}.
 \tag{E.3}\label{eq:general-perturbation}
\end{equation}
Choosing \(\varepsilon=Q_R^{1/3}\) proves
\eqref{eq:general-full-bound}. The correlation matrix trace constraint,
the null lower bound in \(s_R\), and a split between \(m-p\le p\) and
\(m-p>p\) absorb the explicit trace terms into
\(C\{\lambda_{m,p}+p^{-1/3}\}\). This proves
\eqref{eq:general-simple-bound}. The cube
root enters only in this last Chebyshev and anticoncentration balance.
\end{proof}

\section{Auxiliary estimates and normalization checks}
\label{app:auxiliary}

For completeness, we record the finite inequalities used when the preceding
analytic bounds are converted to the normalizations in the theorem
statements. From \eqref{eq:polygamma-series} and monotonicity,
\[
 \frac1x+\frac1{2x^2}
 \le \psi_1(x)
 \le \frac1x+\frac1{x^2},\qquad x>0,                \tag{F.1}
\]
and
\[
 \frac1{x^2}\le-\psi_2(x)
 \le \frac1{x^2}+\frac2{x^3}.                       \tag{F.2}
\]
Indeed, isolate the first summand and bound the tail by the corresponding
integral; applying the trapezoidal correction gives the lower half of
(F.1). Summing (F.1) over \(a_j\) and subtracting its value at \(M\) proves
the positivity of the exact null variance and the comparisons with the
harmonic scale used in (C.1). Applying (F.2) in the same way proves the
third cumulant majorant used in (B.2). No asymptotic interchange is needed
for these steps: all terms are nonnegative after the displayed signs are
extracted.

The leading scale for general \(R\) has a similarly direct comparison. Write
\(A=R-I_p\). Since \(\operatorname{tr}A=0\) for a correlation matrix,
orthogonal invariance makes the null component uncorrelated with
\(\operatorname{tr}(AW_0)\), while
\(\operatorname{Var}\{\operatorname{tr}(AW_0)/m\}
=2\operatorname{tr}(A^2)/m\).  Hence the approximating variance is
\[
 s_R^2=V_{m,p}+\frac{2}{m}\operatorname{tr}(A^2). \tag{F.3}
\]
The exact variance of \(M_R\) differs from (F.3) by
\[
 p\{\psi_1(m/2)-2/m\}.
\]
Inequality (F.1), \(m\ge p\), and the elementary lower bound for
\(V_{m,p}\) give the ratio estimate in (D.9). Thus replacing the actual
standard deviation of \(M_R\) by \(s_R\) costs \(O(p^{-1})\) in
Kolmogorov distance, using the positive affine Gaussian comparison
\[
 d_{\rm K}(aN,N)\le C|a-1|,\qquad a>0.              \tag{F.4}
\]
The proof of (F.4) differentiates \(\Phi(x/a)\) in \(a\) and uses
\(\sup_x|x|\phi(x)<\infty\).

To control the population cubic term, monotonicity of finite dimensional
Schatten norms \cite[Chapter~IV]{Bhatia1997}, namely
\((\sum_i|\lambda_i(A)|^3)^{1/3}
\le(\sum_i\lambda_i(A)^2)^{1/2}\), gives
\[
 \operatorname{tr}|A|^3
 \le \{\operatorname{tr}(A^2)\}^{3/2}.
\]
Combining this inequality with (F.3), and using the null component when the
correlation energy is small, proves the uniform estimate
\[
 \rho_R\le\frac{1}{2^{3/2}\sqrt m}.                 \tag{F.5}
\]
For the nonlinear term, \(\operatorname{tr}(A^2)\le p(p-1)\) and the two
components in (F.3) imply
\[
 Q_R\le \frac4{p-1}+\frac2m.                        \tag{F.6}
\]
Balancing \(C\varepsilon+Q_R/\varepsilon^2\) at
\(\varepsilon=Q_R^{1/3}\), then splitting according as the null variance or
the correlation energy term dominates, gives the simplified
\(p^{-1/3}\) envelope printed in Theorem~\ref{thm:general-full}.

The centering and population standardization at the start of Appendix~D also
prove that the canonical statistic has exactly the original sample law.
Consequently its center is \(\log\det R+b_{m,p}\), not an asymptotic
replacement; this is the arbitrary covariance bridge in
Table~\ref{tab:lean-crosswalk}.
\section{Lean statement crosswalk and proof boundary}
\label{app:lean-crosswalk}\label{app:lean-ledger}\label{app:ledger}

Table~\ref{tab:lean-crosswalk} records the paper facing boundary of the
thirteen principal public theorem endpoints in the published
\texttt{v1.1.2} capsule~\cite{ZhaoLeanVerification2026}.
``Exact/equivalent'' describes the endpoint level relation after notation
expansion or a proved bridge. These endpoints are not the complete list of
supporting declarations in their dependency cone. The verification capsule
records the audited formal route for every numbered equation, including any
explicit remaining hypothesis;
ordinary mathematical prose is not kernel checked.

\begin{table}[h]
\centering
\begingroup
\fontsize{5.8}{6.0}\selectfont
\setlength{\tabcolsep}{3pt}
\renewcommand{\arraystretch}{0.88}
\caption{Statement-level Lean crosswalk.}\label{tab:lean-crosswalk}
\begin{tabular}{p{0.22\textwidth}p{0.50\textwidth}p{0.16\textwidth}}
\toprule
Paper result & Lean endpoint & Relation\tabularnewline
\midrule
Proposition~\ref{prop:beta-product} &
\nolinkurl{map_centeredSampleCorrelationDet_succ_eq_map_product_betaFactors}
& Exact\tabularnewline
Theorem~\ref{thm:null-edgeworth} &
\nolinkurl{uniformNullEdgeworthTarget_proved};
\nolinkurl{tendsto_uniformActualNullSharpKolmogorov_relative_error_zero}
& Exact\tabularnewline
Theorem~\ref{thm:general-leading} &
\nolinkurl{paperTheoremFiveOne_exact}
& Exact/equivalent\tabularnewline
Theorem~\ref{thm:general-full} &
\nolinkurl{paperTheoremFiveTwo_exact};
\nolinkurl{paperTheoremFiveTwo_arbitraryCovariance_exact}
& Exact/equivalent\tabularnewline
Proposition~\ref{prop:asymptotics} &
\nolinkurl{tendsto_nullASeries_div_uniformScale};
\nolinkurl{paperEquationSixThree_exact};
\nolinkurl{tendsto_nullVSeries_div_uniformScale};
\nolinkurl{tendsto_square_nullVSeries_sub_two_log};
\nolinkurl{paperEquationSixSix_exact}
& Exact/equivalent\tabularnewline
Corollary~\ref{cor:square-supremum} &
\nolinkurl{tendsto_scaledNullKolmogorovSup_closed};
\nolinkurl{nullSharpSupremumConstant_decimal8}
& Exact\tabularnewline
\bottomrule
\end{tabular}
\endgroup
\end{table}

The five public endpoints in the Proposition~\ref{prop:asymptotics} row cover
\eqref{eq:A-uniform}--\eqref{eq:A0}. Equation \eqref{eq:Ad} is the formal
definition used by these endpoints rather than a separate theorem endpoint.
The endpoint \nolinkurl{paperEquationSixThree_exact} packages the explicit
tail bounds and growing gap equivalent in \eqref{eq:A-growing};
\nolinkurl{paperEquationSixSix_exact} packages the square constant, recurrence,
and strict decrease in \eqref{eq:A0}.
The optional exact variance identity \eqref{eq:tau-exact} lies outside the
unconditional crosswalk.  Its deterministic Kibble series and bounds are
kernel checked, but in the current Lean release its final probability
identification assumes an explicit scalar Kibble limit certificate.  That
certificate is an ordinary unproved hypothesis, not a project axiom.  This
conditional identity is not used in the formal endpoints for
Theorems~\ref{thm:general-leading} and \ref{thm:general-full}.

The release contains the complete source, toolchain and dependency pins,
verification scripts, and exact statement crosswalk described in
Section~\ref{sec:lean}.  Kernel checking establishes the encoded deductions;
bibliographic priority and informal prose remain scholarly claims.

\endgroup

\begingroup
\def\bibfont{\normalsize\raggedright}
\def\bdoi#1{\href{https://doi.org/#1}{\nolinkurl{https://doi.org/#1}}}
\bibliographystyle{unsrtnat}
\bibliography{references}
\endgroup

\end{document}